\documentclass[12pt]{amsart}
\usepackage{amscd,amsmath,amsthm,amssymb}
\usepackage[left]{lineno}
\usepackage{color}
\usepackage{stmaryrd}
\usepackage[utf8]{inputenc}
\usepackage{cleveref}
\usepackage{graphicx}
\usepackage{epstopdf}

\usepackage{xcolor}
\usepackage[misc,geometry]{ifsym}
\usepackage{tikz}
\usepackage{subcaption}

\definecolor{verylight}{gray}{0.97}
\definecolor{light}{gray}{0.9}
\definecolor{medium}{gray}{0.85}
\definecolor{dark}{gray}{0.6}

 \usepackage{fullpage}

 \def\NZQ{\mathbb}               

 \def\ZZ{{\NZQ Z}}

 \def\KK{{\NZQ K}}
 
 \def\G{{\mathcal G}}

 \def\0b{{\mathbf 0}}

 \def\opn#1#2{\def#1{\operatorname{#2}}} 
 \opn\ara{ara}
 \opn\chara{char}
 \opn\length{\ell}
 \opn\pd{pd}
 \opn\rk{rk}
 \opn\projdim{proj\,dim}
 \opn\injdim{inj\,dim}
 \opn\rank{rank}
 \opn\depth{depth}
 \opn\grade{grade}
 \opn\height{height}
 \opn\embdim{emb\,dim}
 \opn\codim{codim}
 
 \opn\Tr{Tr}
 \opn\bigrank{big\,rank}
 \opn\superheight{superheight}
 \opn\lcm{lcm}
 \opn\trdeg{tr\,deg}
 \opn\reg{reg}
 \opn\lreg{lreg}
 \opn\ini{in}
 \opn\lpd{lpd}
 \opn\size{size}
 \opn\sdepth{sdepth}
 \opn\link{link}
 \opn\fdepth{fdepth}\opn\lex{lex}
 \opn\tr{tr}
 \opn\type{type}
 \opn\gap{gap}
 \opn\arithdeg{arith-deg}
 \opn\HS{HS}
 \opn\GL{GL}
 \opn\supp{supp}
\opn\St{St}

 \opn\div{div} \opn\Div{Div} \opn\cl{cl} \opn\Cl{Cl}
 \opn\Spec{Spec} \opn\Supp{Supp} \opn\supp{supp} \opn\Sing{Sing}
 \opn\Ass{Ass} \opn\Min{Min}\opn\Mon{Mon}
 \opn\Ann{Ann} \opn\Rad{Rad} \opn\Soc{Soc}\opn\Deg{Deg}
 \opn\Im{Im} \opn\Ker{Ker} \opn\Coker{Coker} \opn\Am{Am}
 \opn\Hom{Hom} \opn\Tor{Tor} \opn\Ext{Ext} \opn\End{End}
 \opn\Aut{Aut} \opn\id{id}
 
 \opn\nat{nat}
 \opn\pff{pf}
 \opn\Pf{Pf} \opn\GL{GL} \opn\SL{SL} \opn\mod{mod} \opn\ord{ord}
 \opn\Gin{Gin} \opn\Hilb{Hilb}\opn\sort{sort}
 \opn\PF{PF}\opn\Ap{Ap}
 \opn\mult{mult}
 \opn\bight{bight}
 \opn\aff{aff}
 \opn\relint{relint} \opn\st{st}
 \opn\lk{lk} \opn\cn{cn} \opn\core{core} \opn\vol{vol}  \opn\inp{inp} \opn\nilpot{nilpot}
 \opn\link{link} \opn\star{star}\opn\lex{lex}\opn\set{set}
 \opn\width{wd}
 \opn\Fr{F}
 \opn\QF{QF}
 \opn\G{G}
 \opn\type{type}\opn\res{res}
 \opn\conv{conv}
 \opn\Ind{Ind}
 \opn\gr{gr}
 
 \def\pot#1#2{#1[\kern-0.28ex[#2]\kern-0.28ex]}

 \opn\dirlim{\underrightarrow{\lim}}
 \opn\inivlim{\underleftarrow{\lim}}
 \def\Implies{\ifmmode\Longrightarrow \else
         \unskip${}\Longrightarrow{}$\ignorespaces\fi}
 \def\implies{\ifmmode\Rightarrow \else
         \unskip${}\Rightarrow{}$\ignorespaces\fi}
 \def\iff{\ifmmode\Longleftrightarrow \else
         \unskip${}\Longleftrightarrow{}$\ignorespaces\fi}

 \let\:=\colon
 \newtheorem{Theorem}{Theorem}[section]
 \newtheorem{Lemma}[Theorem]{Lemma}

 \newtheorem{Remark}[Theorem]{Remark}
 
 \newtheorem{Example}[Theorem]{Example}
 
 \newtheorem{Definition}[Theorem]{Definition}

 \let\epsilon\varepsilon
 \let\kappa=\varkappa
 \def\qed{\ifhmode\textqed\fi
       \ifmmode\ifinner\quad\qedsymbol\else\dispqed\fi\fi}
 \def\textqed{\unskip\nobreak\penalty50
        \hskip2em\hbox{}\nobreak\hfil\qedsymbol
        \parfillskip=0pt \finalhyphendemerits=0}
 \def\dispqed{\rlap{\qquad\qedsymbol}}
 
 \opn\dis{dis}
 \def\pnt{{\raise0.5mm\hbox{\large\bf.}}}
 
 \opn\Lex{Lex}

\begin{document}

\title{The $\circ$ operation and $*$   operation of    Cohen-Macaulay bipartite graphs}

\author{Yulong Yang,  Guangjun Zhu$^{\ast}$,  Yijun Cui and Shiya Duan}


\address{School of Mathematical Sciences, Soochow University, Suzhou, Jiangsu, 215006, P. R. China}

\email{ zhuguangjun@suda.edu.cn(Corresponding author:Guangjun Zhu),	\linebreak[4]
		1975992862@qq.com(Yulong Yang), 237546805@qq.com(Yijun Cui),3136566920@qq.com(Shiya Duan).}

\thanks{$^{\ast}$ Corresponding author}



\maketitle
\begin{abstract}
Let $G$  be a  finite simple graph  with the vertex set $V$ and let $I_G$ be its edge ideal in the  polynomial ring $S=\KK[x_V]$. In this paper,  we compute the depth and the Castelnuovo--Mumford regularity of
$S/I_G$ when   $G=G_1\circ G_2$ or $G=G_1* G_2$ is a graph obtained from Cohen-Macaulay bipartite graphs $G_1$, $G_2$ by $\circ$ operation or $*$ operation, respectively.
\end{abstract}

\section{Introduction}
Let $G$ be a simple graph on the vertex set $V(G)$ without any isolated vertices.
Let $\KK$ be  a field and  $S:=\KK[V(G)]=\KK[v|v\in V(G)]$  a polynomial ring with $\deg(v)=1$.   The edge ideal of $G$, denoted by $I_G$, is defined as $I_G:=(uv \mid  \{u, v\}\in E(G))$, where  $E(G)$ is the set of all edges of $G$. Over the past two decades, many authors established connections between the combinatorial properties of $G$ and the algebraic properties of $I_G$  (see
for example \cite{FHM,FM,HT1,Ka,KM,W,Zhu1,Zhu2,ZCYY}).

Let $G$ be a graph  with the  edge set $E(G)$.  A subset $M\subset E(G)$ is a {\em matching} of $G$ if $e\cap e'=\emptyset$  for any
pair of edges $e, e'\in M$. A matching is a  {\em maximal  matching} if it has the greatest possible number of edges. The {\em matching number} of $G$, denoted by $\alpha(G)$, is the   maximum size of a maximal matching in $G$. The {\em minimum matching number} of $G$, denoted by $\beta(G)$, is the  minimum size of a  maximal matching in $G$. An {\em induced matching} in $G$ is a matching $M$ such that the induced subgraph of $G$ over the vertices of $M$ does not contain any edge other than those already in $M$.
The {\em induced matching number} of $G$, denoted by $\vartheta(G)$, is the maximum size of an induced matching in $G$.
 By \cite{Ka} and \cite{W}, we know that for any graph $G$,
\[
\vartheta(G)\le \reg(S/I_G)\le \beta(G).
\]
The first inequality  becomes equality in the following cases: (a) $G$ is a chordal graph; (b) $G$  is a weakly chordal graph; (c) $G$  is a sequentially
Cohen-Macaulay bipartite graph;  (e)  $G$ is a unmixed bipartite graph;   (e)  $G$ is a very well-covered graph;
(f)  $G$ is a $C_5$-free vertex decomposable graph; (g)  $G$   is a $(C_4,C_5)$free vertex decomposable graph (see \cite{BC,HT1,KM,K,MMCRTY,T,W}).

For a graph $G$ with vertex set $V(G)$.
For a vertex  $x$ in $V(G)$, the subgraph $\St(x)$ of $G$ with vertex set $N_G[x]$ and edge set $\{\{x,y\} |y \in N_G(x)\}$ is called a {\em star with center} $x$. A star {\em packing} of $G$
is a family $\mathcal{X}$ of stars in $G$ which are pairwise disjoint, i.e., $V(\St(x))\cap V(\St(x'))=\emptyset$,
for $\St(x), \St(x') \in  \mathcal{X}$ with $x \ne x'$. The quantity $\max\{|\mathcal{X}|\mid \mathcal{X} \text{ is a star packing of }G\}$
is called the {\em star packing number} of $G$, denoted by $\gamma(G)$. Fouli et al. in \cite{FHM} showed  that
\[
\depth(S/I_G)\ge \gamma(G).
\]

Let  $u,v\in V(G)$.  The distance of  $u$ and $v$, denoted by $d(u, v)$, is the length of the shortest path between $u$ and $v$. If $G$ is connected, then
the diameter of $G$ is $d(G)=\max\{d(u, v) | u, v \in V\}$.
 Fouli and Morey in \cite{FM} showed  that if  a graph $G$ has $p$ connected components, then
\[
\depth(S/I_G)\ge  \sum\limits_{i=1}^{p}\lceil\frac{d_i+1}{3}\rceil
\]
where $\lceil\frac{d_i+1}{3}\rceil$ is the smallest integer $\ge \frac{d_i+1}{3}$ and $d_i$ is the diameter of the $i$-th connected component of $G$.  The second author of this paper in \cite{Zhu1} proved that
if $G$ is a path, then $\depth(S/I_G)$ can reach this lower bound.
Morey  et al. in \cite{MV} showed that for a connected  bipartite graph $G$ with $n$ vertices, then
  \[
\depth(S/I_G)\le \lfloor\frac{n}{2}\rfloor
\]
where $\lfloor\frac{n}{2}\rfloor$ is the largest integer $\le \frac{n}{2}$. The second author  of this paper in \cite{Zhu1} and \cite{Zhu2} provided  some exact formulas for the depth and  regularity of the edge ideals of path graphs and cycle graphs respectively.

The first three authors of this article in \cite{ZCYY}  studied two family of simple graphs  obtained from some  fan graphs by the $*$ operation and the $\circ$ operation, respectively. For such two graphs, we gave some formulas for the depth and  regularity of $S/I_G$.

In this article, we are interested in  algebraic properties of  depth and  regularity  of $S/I_G$ if $G$ is a  graph obtained from two  Cohen-Macaulay bipartite graphs by the $\circ$ operation or the $*$ operation.

The article is organized as follows. In section \ref{sec:prelim}, we will recall some basic definitions and terminology that we will need later.  In sections \ref{sec:study},   we will study the depth and regularity  of a  bipartite graph obtained from  a Cohen-Macaulay bipartite graph by deleting its one leaf.
We give some exact formulas for the depth and regularity  of the  edge ideal of such a graph. In sections \ref{sec:operation},
we will study some   graphs  obtained from  Cohen-Macaulay bipartite graphs by the $\circ$ operation or the  $*$ operation. For such graphs, we give some exact formulas for the depth and regularity  of their  edge ideals.

\section{Preliminary}
\label{sec:prelim}

In this section, we gather together the needed definitions and basic facts, which will
be used throughout this paper. However, for more details, we refer the reader to \cite{BM,HH1,V}.

Let $G=(V(G),E(G))$ be a finite simple (no loops, no multiple edges) graph, where $V(G)$  and  $E(G)$ are the vertex set  and edge set of  $G$, respectively. Sometimes for short we denote  $V(G)$ and  $E(G)$
by $V$ and $E$ respectively. The {\em neighborhood}  of a vertex $v$ in $G$ is defined as  $N_G(v)=\{u\,|\, \{u,v\}\in E(G)\}$
 and its degree, denoted by $\deg_G(v)$, is $|N_G(v)|$. If $|N_G(v)|=1$, then  $v$ is called a leaf.
Set $N_G[v]=N_G(v)\cup\{v\}$.  For $A\subset V(G)$, $G|_A$ denotes the {\em induced subgraph} of $G$ on  the  set $A$, i.e., for $i,j \in A$, $\{i,j\} \in E(G|_A)$ if and only if $\{i,j\}\in E(G)$. For $W\subseteq V(G)$, we denote by $G\backslash W$ the induced subgraph
of  $G$  on $V(G) \setminus W$. For a vertex $v\in V(G)$,  we denote by $G\backslash v$  the induced subgraph of $G$ on the  set $V(G)\backslash \{v\}$ for simplicity.

A {\em walk} of length $(n-1)$ in a graph $G$ is an alternating sequence of vertices and edges $w=\{v_1, z_1, v_2,\ldots, v_{n-1}, z_{n-1}, v_{n}\}$, where $z_i=\{v_{i}, v_{i+1}\}$ is the edge
joining $v_i$ and $v_{i+1}$. A walk is closed if $v_1=v_{n}$. A walk may also be denoted $\{v_1,\ldots, v_n\}$, the edges being evident by context. A {\em cycle} of length $n$ is a
closed walk, in which the points $v_1,\ldots,v_n$ are distinct. We denote  the graph consisting of a cycle with $n$ vertices by $C_n$. A {\em path} is a walk with all the points distinct. For simplicity, a {\em path} with $n$ vertices, denoted $P_n$, is a walk with
the vertex set $[n]$ and edge set $\{\{1,2\},\{2,3\},\ldots,\{n-1,n\}\}$, and the length of $P_n$ is defined
to be $n-1$.
Any graph isomorphic to $P_n$ is also called a path.

In the sequel, let $S_+$ be the unique graded maximal ideal of the standard
graded algebra $S$. The local cohomology modules of a finitely generated
graded $S$-module $M$ with respect to $S_+$ are denoted by $H_{S_+}^i(M)$
for $i\in \ZZ$.

\begin{Definition}
    Let $M$ be a finitely generated graded $S$-module.
    \begin{enumerate}
        \item[(a)] The \emph{depth} of $M$ is defined as
            \[
                \depth(M):= \min\{i: H_{S_+}^i(M)\ne 0\}.
            \]
        \item For $i = 0, \dots , \dim(M)$, the $i$\textsuperscript{th} \emph{$a$-invariant} of $M$ is defined as
            \[
                a_i(M):=\max\{t : (H_{S_+}^i(M))_t \ne 0\}
            \]
            with the convention that $\max \emptyset = -\infty$.
        \item The \emph{Castelnuovo--Mumford regularity} of $M$ is defined as
            \[
                \reg(M):=\max\{a_i(M) + i:0\le i\le \dim(M)\}.
            \]
    \end{enumerate}
\end{Definition}
A graph $G$ is called  Cohen-Macaulay (abbreviated as C-M) if the  quotient ring $S/I_G$ is Cohen-Macaulay, i.e., $\depth(S/I_G)=\dim(S/I_G)$.
For  a proper non-zero homogeneous ideal $I$ in $S$, it is known that $\reg(S/I)=\reg(I)-1$.

The following lemmas are often used  to compute  the depth and regularity of a module.
 In particular, since the facts in Lemma \ref{lem:direct_sum} are well-known, they will be used implicitly in this paper.

\begin{Lemma}
	\label{lem:direct_sum}
	Let $M,N$ be two  finitely generated graded $S$-modules. Then,
	\begin{itemize}
		\item[(1)] 
		$\depth(M\oplus N)=\min\{\depth(M),\depth(N)\}$, and
		\item[(2)] 
		$\reg(M\oplus N)=\max\{\reg(M),\reg(N)\}$.
	\end{itemize}
\end{Lemma}

\begin{Lemma} {\em (\cite[Lemmas 2.1 and 3.1]{HT})}
	\label{exact}
	Let $0\longrightarrow M\longrightarrow N\longrightarrow P\longrightarrow 0$  be an exact sequence of finitely generated graded $S$-modules. Then we have
	\begin{enumerate}
	\item 	\label{exact-2}$\depth\,(M)\geq \min\{\depth\,(N), \depth\,(P)+1\}$, the equality holds if $\depth\,(N)\\ \neq \depth\,(P)$.
	\item 	\label{exact-4} $\reg\,(M)\leq \max\{\reg\,(N), \reg\,(P)+1\}$, the equality holds if $\reg\,(N) \neq \reg\,(P)$.
\end{enumerate}
\end{Lemma}


\begin{Lemma}{\em (\cite[Lemma 2.2, Lemma 3.2]{HT})}
	\label{sum}
	Let $S_{1}=\KK[x_{1},\dots,x_{m}]$ and $S_{2}=\KK[x_{m+1},\dots,x_{n}]$ be two polynomial rings  over $\KK$, let $I\subset S_{1}$ and
	$J\subset S_{2}$ be two non-zero homogeneous  ideals. Let $S=S_1\otimes_\KK S_2$.  Then we have
	\begin{enumerate}
		\item \label{sum-1} $\reg\,(S/(I+J))=\reg\,(S_1/I)+\reg\,(S_2/J)$;
		\item \label{sum-2} $\depth\,(S/(I+J))=\depth\,(S_1/I)+\depth\,(S_2/J)$;
	\end{enumerate}
\end{Lemma}

For a subset $A\subset V(G)$, let $(A)=(v\mid v\in A)$ be an ideal of $S=\KK[V]$ generated by the element in $A$. The following lemma is very important for the whole paper.
\begin{Lemma}
	\label{decomposition}{\em (\cite[Lemma 1.5]{ZCYY})}
	Let $G=(V,E)$ be  a connected simple graph. Let $J=(N_G(v))+I_{G \backslash N_G[v]}$ and $K=(v)+I_{G \backslash v}$, where $v\in V$. Then
	\begin{enumerate}
		\item  \label{decomposition-1} $J+K=(N_G[v])+I_{G \backslash N_G[v]}$;
		\item \label{decomposition-2} $I_G=J\cap K$;
		\item \label{decomposition-3} $\depth(S/J)=\depth(S/(J+K))+1$;
       \item \label{decomposition-4} $\reg(S/J)=\reg(S/(J+K))$.
\end{enumerate}
\end{Lemma}

A graph $G$ is called \emph{bipartite} if there  exists  a  \emph{bipartition} $V(G)=V_1\sqcup V_2$ with $V_1\cap V_2=\emptyset$ such that each edge of $G$ is of
the form $\{i,j\}$ with $i\in V_1$ and $j\in V_2$.
For  a positive integer $n$, let  $[n]=\{1,2,\ldots, n\}$ by convention. In \cite{HH2},
Herzog and Hibi classified all C-M bipartite graphs. we state their result.
\begin{Theorem}
  {\em (\cite[Theorem 3.4]{HH2})}
    \label{cm bipart}
Let $G=(V(G),E(G))$ be a  bipartite graph with bipartition $V(G)=\{x_1,x_2,\dots,x_n\} \sqcup \{y_1,y_2,\dots,y_m\}$. Then $G$ is C-M if and only if $n=m$, and there exists a labeling such that
\begin{enumerate}
 \item   \label{cm bipart-1} $\{x_i,y_i\} \in E(G)$ for $i\in [n]$,
 \item   \label{cm bipart-2} if $\{x_i,y_j\} \in E(G)$, then $i \le j$, and
\item  \label{cm bipart-3} if $\{x_i,y_j\}\in E(G)$ and $\{x_j,y_k\}\in E(G)$ with $i<j<k$, then $\{x_i,y_k\} \in E(G)$.
    \end{enumerate}
\end{Theorem}
 By Theorem \ref{cm bipart},  the vertices $y_1$ and $x_n$ must must be of degree one, and their neighborhood points are $x_1$ and $y_n$,  respectively.
  Let $N_G(y_n)=\{x_{i_1},\ldots, x_{i_s},x_n\}$ for some $x_{i_j}\in \{x_1,\ldots,x_n\}$. Francisco et al. in \cite{FHT} showed:

\begin{Lemma} {\em (\cite[Lemma 3.4]{FHT})}
 \label{cm subgraph}Let $G$ be  a C-M bipartite graph  with bipartition $V(G)=\{x_1,\dots,x_n\} \sqcup \{y_1,\dots,y_n\}$. Then
\begin{enumerate}
 \item   \label{cm subgraph-1} $G\backslash\{x_n,y_n\}$ is a C-M bipartite graph.
 \item   \label{cm subgraph-2} $G\backslash\{x_{i_1},y_{i_1},\ldots,x_{i_s},y_{i_s},x_n,y_n\}$ is a C-M bipartite graph.
    \end{enumerate}
 \end{Lemma}

For a  proper ideal $I\subset  S$, its \emph{arithmetic rank},  denoted by $\ara(I)$, is  the minimum number of elements of
$S$ that generate an ideal whose radical is $I$.  An ideal   is said to be a \emph{set-theoretic complete intersection} if its arithmetic rank  is equal to its height.
In general, if $I$ is a  square-free   monomial ideal, we have the well-known inequalities
\[
\height(I)\le \pd(S/I)\le \ara(I)
\]
where  $\height(I)$ is the height of $I$  and $\pd(S/I)$ is the projective dimension of the quotient ring $S/I$.

\begin{Lemma} 
 \label{depth-reg}Let  $G=(V(G),E(G))$ be  a C-M bipartite graph without isolated vertices. Then
\begin{enumerate}
 \item   \label{depth-reg-1} $\depth(S/I_{G})=\frac{|V(G)|}{2}$;
 \item   \label{depth-reg-2}  $\reg(S/I_{G})=\vartheta(G)$, where $\vartheta(G)$ is the induced matching number of $G$.
    \end{enumerate}
 \end{Lemma}
\begin{proof} (1) Since $G$ is  a C-M bipartite graph,    $I_G$ is  unmixed and a set-theoretic complete intersection by \cite[Corollary 3.5]{MA}. This forces $\pd(S/I_G)=\height(I_G)=\frac{|V(G)|}{2}$.
 It follows  from the graded Auslander–Buchsbaum formula that $\depth(S/I_{G})=|V(G)|-\pd(S/I_G)=\frac{|V(G)|}{2}$.

(2)  is a direct consequence  of \cite[Corollary 3.4]{T}.
\end{proof}

\begin{Lemma}{\em (\cite[Theorem 3.3, Corollary 3.3]{Zhu1})}
	\label{path}
 Let $n\ge 2$ be an integer and $P_n$  be  a path  with $n$ vertices, then
 \[
 \depth(S/I_G)= \lceil\frac{n}{3}\rceil  \text{\ and\ }  \reg(S/I_G)=\lfloor\frac{n+1}{3}\rfloor.
 \]
\end{Lemma}

\medskip
\section{Study of   bipartite graphs}
\label{sec:study}
In this section,  we will study the depth and regularity  of a  bipartite graph obtained from  a Cohen-Macaulay bipartite graph by deleting its one leaf.
We give some exact formulas for the depth and regularity  of the  edge ideal of such a graph.

\begin{Lemma}
\label{subgraph-depth}
Let  $G$ be a  C-M bipartite graph with  a leaf $u$. Let $N_G(u)=\{v\}$  with $\deg_G(v)\ge 2$.  Then
 \[
\depth(S_{G\backslash u}/I_{G\backslash u})=\depth(S_G/I_G)-1.
\]
\end{Lemma}
\begin{proof}Let  $V(G)=\{x_1,x_2,\dots,x_n\} \sqcup \{y_1,y_2,\dots,y_n\}$ and $u=x_i$ for some $i\in [n]$, by symmetry. Then $v=y_i$ with $\deg_G(y_i)\ge 2$. Thus $n\ge 2$.
We prove the claimed formula  by  induction on $n$.
If $n=2$, then $G$ and $G\backslash u$ are paths with $4$ and $3$ vertices respectively. This case is covered by	Lemma \ref{path}.

In the following, we assume that $n\ge 3$. Let $H=G\backslash u$ and  $N_H(v)=\{x_{i_1},\ldots, x_{i_s}\}$ with  $1\leq i_1< i_2 < \dots < i_s<i$, then  $\deg_H(y_{i_1})=1$ by Theorem  \ref{cm bipart}(\ref{cm bipart-2}).
Let $N_H(x_{i_1})=\{y_{j_1},y_{j_2},\dots,y_{j_t}\}$ with $j_1=i_1<j_2 < \dots < j_t \leq n $. Set $J=(N_H(x_{i_1}))+I_{H \backslash N_H[x_{i_1}]}$,  $K=(x_{i_1})+I_{H \backslash x_{i_1}}$.
 We  distinguish between the following two cases:

(1)   If $|N_H(v)|=1$, then $s=1$ and   $H \backslash x_{i_1}$   is the disjoint union of $G\backslash  \{u,v,x_{i_1},y_{i_1}\}$ and the isolated set $\{y_{i_1},v\}$.
Thus, by Lemma \ref{depth-reg}(\ref{depth-reg-1}), we have
\[
\depth(S_H/K)=2+(n-2)=n.
\]
Meanwhile,   $H\backslash N_H[x_{i_1}]$  has one of the following forms. Figure $1$ will be helpful in understanding the arguments.
\begin{enumerate}
 \item[(a)]  $H\backslash N_H[x_{i_1}]=G\backslash  \{u,v,x_{i_1},y_{i_1}\}$;
  \item[(b)] $H\backslash N_H[x_{i_1}]$ is the disjoint union of $G\backslash  \{x_{j_1},y_{j_1}, \ldots,x_{j_t},y_{j_t}\}$ and the isolated set  $\{x_{j_2},\dots,x_{j_t}\}\backslash\{u\}$.
  \end{enumerate}
  By Lemma  \ref{cm subgraph}(\ref{cm subgraph-2}), we get that both
 $G\backslash  \{u,v,x_{i_1},y_{i_1}\}$ and  $G\backslash  \{x_{j_1},y_{j_1}, \ldots,x_{j_t},y_{j_t}\}$ are  C-M bipartite graphs. We consider the following two subcases:

 (i) If $H\backslash N_H[x_{i_1}]$  is of form (a), then we get by Lemma \ref{depth-reg}(\ref{depth-reg-1}) that
 \[
 \depth(S_H/J)=1+\depth(S_{H\backslash N_H[x_{i_1}]}/I_{H\backslash N_H[x_{i_1}]})=1+(n-2)=n-1.
 \]
  (ii) If $H\backslash N_H[x_{i_1}]$  is of form (b), then we also have
  \begin{align*}
	\depth(S_H/J)=1+\depth(S_{H\backslash N_H[x_{i_1}]}/I_{H\backslash N_H[x_{i_1}]})
=1+(t-2)+(n-t)=n-1.
\end{align*}

(2) If $|N_H(v)|\ge 2$, then $H \backslash x_{i_1}$ is the disjoint union of $G\backslash  \{x_{i_1},y_{i_1},u\}$ and an isolated vertex $y_{i_1}$.
Note that  $G\backslash  \{x_{i_1},y_{i_1},u\}$ can be  viewed  as $G_1\backslash u$, where $G_1=G\backslash  \{x_{i_1},y_{i_1}\}$. In this case, let $H'=G_1 \backslash u$, then $|N_{H'}(v)|=|N_H(v)|-1$. Thus, by induction and
 Lemma \ref{depth-reg}(\ref{depth-reg-1}), we have
\begin{align*}
 \depth(S_H/K)&=1+\depth(S_{G\backslash \{x_{i_1},y_{i_1},u\}}/I_{G\backslash \{x_{i_1},y_{i_1},u\}})\\
 &=1+\depth(S_{G_2}/I_{G_2})-1
 =\depth(S_{G_2}/I_{G_2})=n-1.
\end{align*}
At the same time, $H\backslash N_H[x_{i_1}]$  has one of the following forms.  Figure $2$ will be helpful in understanding the arguments.
\begin{enumerate}
 \item[(c)] $H\backslash N_H[x_{i_1}]$ is the disjoint union of $G\backslash  \{x_{j_1}, \ldots,x_{j_t},y_{j_1},\ldots,y_{j_t}\}$  and the isolated set  $\{x_{j_2}, \dots,x_{j_t}\}\backslash\{u\}$;
     \item[(d)]  $H\backslash N_H[x_{i_1}]=G\backslash  \{u,v,x_{i_1},y_{i_1}\}$.
  \end{enumerate}
As shown in $(a)$ and $(b)$ of  case (1)  above, we can get
\[
 \depth(S_H/J)=1+\depth(S_{H\backslash N_H[x_{i_1}]}/I_{H\backslash N_H[x_{i_1}]})=1+(n-2)=n-1.
 \]

 Furthermore,  using  Lemmas \ref{lem:direct_sum}(1),  \ref{decomposition}(\ref{decomposition-2}), \ref{exact}(\ref{exact-2}) and the following exact sequence
\begin{equation}
	0\longrightarrow \frac{S_H}{J \cap K}\longrightarrow \frac{S_H}{J} \oplus \frac{S_H}{K}\longrightarrow \frac{S_H}{J+K} \longrightarrow 0,
 \label{eqn:SES-1}
\end{equation}	
we get the expected result.
\end{proof}

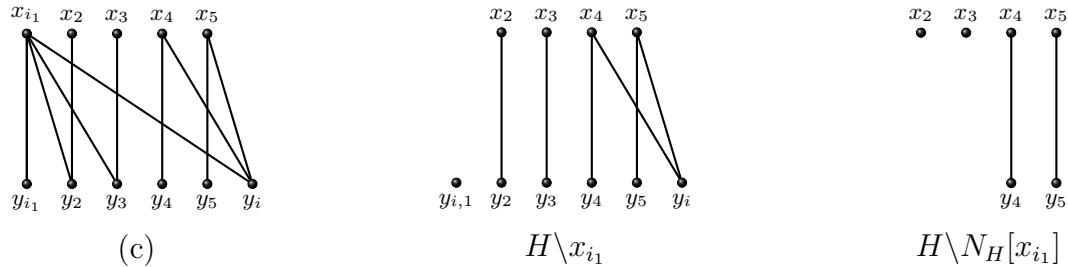
\vspace{0.5cm}
\begin{figure}[!htb]
                \begin{minipage}{0.23\textwidth}
                    \centering
                    \begin{tikzpicture}[thick, scale=1.0, every node/.style={scale=0.98}]]
                       \draw[solid](0,2)--(0,0);
                      \draw[solid](0,2)--(2.8,0);

                        \draw[solid](0.7,2)--(0.7,0);
                      \draw[solid](0.7,2)--(1.4,0);
                       \draw[solid](0.7,2)--(2.1,0);

                         \draw[solid](1.4,2)--(1.4,0);
                       \draw[solid](1.4,2)--(2.1,0);

                       \draw[solid](2.1,2)--(2.1,0);

                        \shade [shading=ball, ball color=black] (0,2) circle (.07) node [above] {\scriptsize$x_{i_1}$};
                        \shade [shading=ball, ball color=black] (0.7,2) circle (.07) node [above] {\scriptsize$x_2$};
                        \shade [shading=ball, ball color=black] (1.4,2) circle (.07) node [above] {\scriptsize$x_3$};
                         \shade [shading=ball, ball color=black] (2.1,2) circle
                         (.07) node [above] {\scriptsize$x_4$};

                        \shade [shading=ball, ball color=black] (0,0) circle (.07) node [below] {\scriptsize$y_{i_1}$};
                        \shade [shading=ball, ball color=black] (0.7,0) circle (.07) node [below] {\scriptsize$y_2$};
                        \shade [shading=ball, ball color=black] (1.4,0) circle (.07) node [below] {\scriptsize$y_3$};
                        \shade [shading=ball, ball color=black] (2.1,0) circle (.07) node [below] {\scriptsize$y_4$};
                        \shade [shading=ball, ball color=black] (2.8,0) circle (.07) node [below] {\scriptsize$y_i$};

                    \end{tikzpicture}
                 \centerline{(a)}
                \end{minipage}\hfill
       \begin{minipage}{0.23\textwidth}
                    \centering
                    \begin{tikzpicture}[thick, scale=1.0, every node/.style={scale=0.98}]]
                      \draw[solid](0,2)--(0,0);
                      \draw[solid](0,2)--(0.7,0);
                      \draw[solid](0,2)--(1.4,0);
                       \draw[solid](0.7,2)--(0.7,0);
                       \draw[solid](0.7,2)--(1.4,0);
                       \draw[solid](1.4,2)--(1.4,0);

                       \shade [shading=ball, ball color=black] (0,2) circle (.07) node [above] {\scriptsize$x_2$};
                       \shade [shading=ball, ball color=black] (0.7,2) circle (.07) node [above] {\scriptsize$x_3$};
                       \shade [shading=ball, ball color=black] (1.4,2) circle (.07) node [above] {\scriptsize$x_4$};

                       \shade [shading=ball, ball color=black] (0,0) circle (.07)node [below] {\scriptsize$y_2$};
                       \shade [shading=ball, ball color=black] (0.7,0) circle (.07)node [below] {\scriptsize$y_3$};
                       \shade [shading=ball, ball color=black] (1.4,0) circle (.07)node [below] {\scriptsize$y_4$};
                    \end{tikzpicture}
                    \centerline{$H\backslash N_H[x_{i_1}]$}
        \end{minipage}\hfill
        \begin{minipage}{0.25\textwidth}
                    \centering
                    \begin{tikzpicture}[thick, scale=1.0, every node/.style={scale=0.98}]]
                     \draw[solid](0,2)--(0,0);
                      \draw[solid](0,2)--(1.2,0);
                      \draw[solid](0,2)--(3,0);

                        \draw[solid](0.6,2)--(0.6,0);
                      \draw[solid](0.6,2)--(1.2,0);
                       \draw[solid](0.6,2)--(1.8,0);
                        \draw[solid](0.6,2)--(2.4,0);

                         \draw[solid](1.2,2)--(1.2,0);

                       \draw[solid](1.8,2)--(1.8,0);

                        \draw[solid](2.4,2)--(2.4,0);

                        \shade [shading=ball, ball color=black] (0,2) circle (.07) node [above] {\scriptsize$x_{i_1}$};
                        \shade [shading=ball, ball color=black] (0.6,2) circle (.07) node [above] {\scriptsize$x_2$};
                        \shade [shading=ball, ball color=black] (1.2,2) circle (.07) node [above] {\scriptsize$x_3$};
                         \shade [shading=ball, ball color=black] (1.8,2) circle
                         (.07) node [above] {\scriptsize$x_4$};
                         \shade [shading=ball, ball color=black] (2.4,2) circle
                         (.07) node [above] {\scriptsize$x_5$};

                        \shade [shading=ball, ball color=black] (0,0) circle (.07) node [below] {\scriptsize$y_{i_1}$};
                        \shade [shading=ball, ball color=black] (0.6,0) circle (.07) node [below] {\scriptsize$y_2$};
                        \shade [shading=ball, ball color=black] (1.2,0) circle (.07) node [below] {\scriptsize$y_3$};
                        \shade [shading=ball, ball color=black] (1.8,0) circle (.07) node [below] {\scriptsize$y_4$};
                        \shade [shading=ball, ball color=black] (2.4,0) circle (.07) node [below] {\scriptsize$y_5$};
                        \shade [shading=ball, ball color=black] (3,0) circle (.07) node [below] {\scriptsize$y_i$};
                    \end{tikzpicture}
                    \centerline{(b)}
        \end{minipage}\hfill
           \begin{minipage}{0.24\textwidth}
                    \centering
                    \begin{tikzpicture}[thick, scale=1.0, every node/.style={scale=0.98}]]
                        \draw[solid](4.2,2)--(4.2,0);
                       \draw[solid](4.2,2)--(5.6,0);
                       \draw[solid](4.2,2)--(6.3,0);

                         \draw[solid](5.6,2)--(5.6,0);

                       \draw[solid](6.3,2)--(6.3,0);

                       \shade [shading=ball, ball color=black] (4.2,2) circle (.07) node [above] {\scriptsize$x_2$};
                       \shade [shading=ball, ball color=black] (4.9,2) circle (.07) node [above] {\scriptsize$x_3$};
                       \shade [shading=ball, ball color=black] (5.6,2) circle (.07) node [above] {\scriptsize$x_4$};
                        \shade [shading=ball, ball color=black] (6.3,2) circle (.07) node [above] {\scriptsize$x_5$};

                       \shade [shading=ball, ball color=black] (4.2,0) circle (.07)node [below] {\scriptsize$y_2$};
                       \shade [shading=ball, ball color=black] (5.6,0) circle (.07)node [below] {\scriptsize$y_4$};
                       \shade [shading=ball, ball color=black] (6.3,0) circle (.07)node [below] {\scriptsize$y_5$};
                    \end{tikzpicture}
                    \centerline{$H\backslash N_H[x_{i_1}]$}
        \end{minipage}\hfill
          \caption{}
            \end{figure}
\begin{figure}[!htb]
    \begin{minipage}{0.25\textwidth}
        \centering
        \begin{tikzpicture}[thick, scale=1.0, every node/.style={scale=0.98}]]
          \draw[solid](0,2)--(0,0);
                      \draw[solid](0,2)--(0,0);
                      \draw[solid](0,2)--(0.6,0);
                      \draw[solid](0,2)--(1.2,0);
                      \draw[solid](0,2)--(3,0);

                        \draw[solid](0.6,2)--(0.6,0);

                       \draw[solid](1.2,2)--(1.2,0);

                       \draw[solid](1.8,2)--(1.8,0);
                       \draw[solid](1.8,2)--(3,0);

                       \draw[solid](2.4,2)--(2.4,0);
                        \draw[solid](2.4,2)--(3,0);

                        \shade [shading=ball, ball color=black] (0,2) circle (.07) node [above] {\scriptsize$x_{i_1}$};
                        \shade [shading=ball, ball color=black] (0.6,2) circle (.07) node [above] {\scriptsize$x_2$};
                        \shade [shading=ball, ball color=black] (1.2,2) circle (.07) node [above] {\scriptsize$x_3$};
                         \shade [shading=ball, ball color=black] (1.8,2) circle
                         (.07) node [above] {\scriptsize$x_4$};
                         \shade [shading=ball, ball color=black] (2.4,2) circle
                         (.07) node [above] {\scriptsize$x_5$};

                        \shade [shading=ball, ball color=black] (0,0) circle (.07) node [below] {\scriptsize$y_{i_1}$};
                        \shade [shading=ball, ball color=black] (0.6,0) circle (.07) node [below] {\scriptsize$y_2$};
                        \shade [shading=ball, ball color=black] (1.2,0) circle (.07) node [below] {\scriptsize$y_3$};
                        \shade [shading=ball, ball color=black] (1.8,0) circle (.07) node [below] {\scriptsize$y_4$};
                        \shade [shading=ball, ball color=black] (2.4,0) circle (.07) node [below] {\scriptsize$y_5$};
                        \shade [shading=ball, ball color=black] (3,0) circle (.07) node [below] {\scriptsize$y_i$};

        \end{tikzpicture}
        \centerline{(c)}
    \end{minipage}\hfill
     \begin{minipage}{0.25\textwidth}
                    \centering
                    \begin{tikzpicture}[thick, scale=1.0, every node/.style={scale=0.98}]]

                      \draw[solid](0.6,2)--(0.6,0);
                      \draw[solid](1.2,2)--(1.2,0);

                      \draw[solid](1.8,2)--(1.8,0);
                      \draw[solid](1.8,2)--(3,0);

                       \draw[solid](2.4,2)--(2.4,0);
                      \draw[solid](2.4,2)--(3,0);

                         \shade [shading=ball, ball color=black] (0.6,2) circle (.07) node [above] {\scriptsize$x_2$};
                        \shade [shading=ball, ball color=black] (1.2,2) circle (.07) node [above] {\scriptsize$x_3$};
                        \shade [shading=ball, ball color=black] (1.8,2) circle (.07) node [above] {\scriptsize$x_4$};
                        \shade [shading=ball, ball color=black] (2.4,2) circle (.07) node [above] {\scriptsize$x_5$};

                          \shade [shading=ball, ball color=black] (0,0) circle (.07) node [below] {\scriptsize$y_{i,1}$};
                    \shade [shading=ball, ball color=black] (0.6,0) circle (.07) node [below] {\scriptsize$y_2$};
                     \shade [shading=ball, ball color=black] (1.2,0) circle (.07) node [below] {\scriptsize$y_3$};
                      \shade [shading=ball, ball color=black] (1.8,0) circle (.07) node [below] {\scriptsize$y_4$};
                       \shade [shading=ball, ball color=black] (2.4,0) circle (.07) node [below] {\scriptsize$y_5$};
                         \shade [shading=ball, ball color=black] (3,0) circle (.07) node [below] {\scriptsize$y_i$};
                    \end{tikzpicture}
                    \centerline{$H\backslash x_{i_1}$}
        \end{minipage}\hfill
        \begin{minipage}{0.24\textwidth}
                    \centering
                    \begin{tikzpicture}[thick, scale=1.0, every node/.style={scale=0.98}]]
                      \draw[solid](1.2,2)--(1.2,0);
                      \draw[solid](1.8,2)--(1.8,0);

                       \shade [shading=ball, ball color=black] (0,2) circle (.07) node [above] {\scriptsize$x_2$};
                           \shade [shading=ball, ball color=black] (0.6,2) circle (.07) node [above] {\scriptsize$x_3$};
                            \shade [shading=ball, ball color=black] (1.2,2) circle (.07) node [above] {\scriptsize$x_4$};
                                \shade [shading=ball, ball color=black] (1.8,2) circle (.07) node [above] {\scriptsize$x_5$};

                             \shade [shading=ball, ball color=black] (1.2,0) circle (.07) node [below] {\scriptsize$y_4$};
                            \shade [shading=ball, ball color=black] (1.8,0) circle (.07) node [below] {\scriptsize$y_5$};
                    \end{tikzpicture}
                    \centerline{$H\backslash N_H[x_{i_1}]$}
        \end{minipage}\hfill
\end{figure}
 \begin{figure}[!htb]
                \begin{minipage}{0.25\textwidth}
                    \centering
                    \begin{tikzpicture}[thick, scale=1.0, every node/.style={scale=0.98}]]

                    \draw[solid](0,-1)--(0,-3);
                   \draw[solid](0,-1)--(2.7,-3);

                        \draw[solid](0.7,-1)--(0.7,-3);
                         \draw[solid](0.7,-1)--(1.4,-3); \draw[solid](0.7,-1)--(2.7,-3);

                       \draw[solid](1.4,-1)--(1.4,-3);
                       \draw[solid](1.4,-1)--(2.7,-3);

                       \draw[solid](2.1,-1)--(2.1,-3);
                       \draw[solid](2.1,-1)--(2.7,-3);

                     \shade [shading=ball, ball color=black] (0,-1) circle
                         (.07) node [above] {\scriptsize$x_{i_1}$};
                         \shade [shading=ball, ball color=black] (0.7,-1) circle
                         (.07) node [above] {\scriptsize$x_2$};
                         \shade [shading=ball, ball color=black] (1.4,-1) circle
                         (.07) node [above] {\scriptsize$x_3$};
                         \shade [shading=ball, ball color=black] (2.1,-1) circle
                         (.07) node [above] {\scriptsize$x_4$};

                      \shade [shading=ball, ball color=black] (0,-3) circle (.07) node [below] {\scriptsize$y_{i_1}$};
                           \shade [shading=ball, ball color=black] (0.7,-3) circle (.07) node [below] {\scriptsize$y_2$};
                            \shade [shading=ball, ball color=black] (1.4,-3) circle (.07) node [below] {\scriptsize$y_3$};
                             \shade [shading=ball, ball color=black] (2.1,-3) circle (.07) node [below] {\scriptsize$y_4$};
                              \shade [shading=ball, ball color=black] (2.7,-3) circle (.07) node [below] {\scriptsize$y_i$};

                    \end{tikzpicture}
                    \subcaption*{$(d)$}
                \end{minipage}\hfill
              \begin{minipage}{0.24\textwidth}
                    \centering
                    \begin{tikzpicture}[thick, scale=1.0, every node/.style={scale=0.98}]]
                          \draw[solid](0.7,2)--(0.7,0);
                      \draw[solid](0.7,2)--(1.4,0);
                      \draw[solid](0.7,2)--(2.7,0);

                        \draw[solid](1.4,2)--(1.4,0);
                        \draw[solid](1.4,2)--(2.7,0);

                        \draw[solid](2.1,2)--(2.1,0);
                        \draw[solid](2.1,2)--(2.7,0);

                          \shade [shading=ball, ball color=black] (0.7,2) circle
                         (.07) node [above] {\scriptsize$x_2$};
                         \shade [shading=ball, ball color=black] (1.4,2) circle
                         (.07) node [above] {\scriptsize$x_3$};
                         \shade [shading=ball, ball color=black] (2.1,2) circle
                         (.07) node [above] {\scriptsize$x_4$};

                         \shade [shading=ball, ball color=black] (0,0) circle (.07) node [below] {\scriptsize$y_{i_1}$};
                    \shade [shading=ball, ball color=black] (0.7,0) circle (.07) node [below] {\scriptsize$y_2$};
                     \shade [shading=ball, ball color=black] (1.4,0) circle (.07) node [below] {\scriptsize$y_3$};
                     \shade [shading=ball, ball color=black] (2.1,0) circle (.07) node [below] {\scriptsize$y_4$};
                      \shade [shading=ball, ball color=black] (2.7,0) circle (.07) node [below] {\scriptsize$y_i$};
                    \end{tikzpicture}
                    \centerline{$H\backslash x_{i_1}$}
        \end{minipage}\hfill
         \begin{minipage}{0.24\textwidth}
                    \centering
                    \begin{tikzpicture}[thick, scale=1.0, every node/.style={scale=0.98}]]
                        \draw[solid](0,2)--(0,0);
                        \draw[solid](0,2)--(0.6,0);

                        \draw[solid](0.6,2)--(0.6,0);

                        \draw[solid](1.2,2)--(1.2,0);

                     \shade [shading=ball, ball color=black] (0,2) circle
                         (.07) node [above] {\scriptsize$x_2$};
                     \shade [shading=ball, ball color=black] (0.6,2) circle
                         (.07) node [above] {\scriptsize$x_3$};
                     \shade [shading=ball, ball color=black] (1.2,2) circle
                         (.07) node [above] {\scriptsize$x_4$};

                    \shade [shading=ball, ball color=black] (0,0) circle (.07) node [below] {\scriptsize$y_2$};
                    \shade [shading=ball, ball color=black] (0.6,0) circle (.07) node [below] {\scriptsize$y_3$};
                    \shade [shading=ball, ball color=black] (1.2,0) circle (.07) node [below] {\scriptsize$y_4$};
                    \end{tikzpicture}
                    \centerline{$H\backslash N_H[x_{i_1}]$}
        \end{minipage}\hfill
                \caption{}
            \end{figure}

\vspace{2cm}

\begin{Remark}
Let  $G$ be a  C-M bipartite graph with  a leaf $u$. Let $N_G(u)=\{v\}$  with  $\deg_G(v)=1$. Then $G\backslash u$ is the disjoint union of $G\backslash \{u,v\}$ and an isolated vertex $v$. So $\depth(S_{G\backslash u}/I_{G\backslash u})=1+\depth(S_{G\backslash \{u,v\}}/I_{G\backslash \{u,v\}})=\depth(S_G/I_G)$.
\end{Remark}

\begin{Lemma}{\em (\cite[Lemma 3.5]{BBH})}
	\label{lem:induced-reg} Let $G$ be a simple graph and $H$ be its induced subgraph. Then $\reg(I_H)\le \reg(I_G)$.
\end{Lemma}

\begin{Lemma}
	\label{reg of S/J-1}
	Let  $G=(V,E)$ be a  C-M bipartite graph with  a leaf $u$. Let $N_G(u)=\{v\}$, $J=(N_G(v))+I_{G \backslash N_G[v]}$ and  $K=(v)+I_{G \backslash v}$. Then
	\begin{enumerate}
		\item[(1)]  $\reg(S_G/J) \leq \reg(S_G/I_{G})-1$;
		\item[(2)] $\reg(S_G/K) \leq \reg(S_G/I_{G})$.
	\end{enumerate}
\end{Lemma}

\begin{proof} It is clear that $\reg(S_G/K)\leq \reg(S_G/I_{G})$ and   $\reg(S_G/J)\le \reg(S_G/I_{G})$ by  Lemma \ref{lem:induced-reg}, since $\reg(S_G/K)=\reg(S_{G \backslash v}/I_{G \backslash v})$, $\reg(S_G/J)=\reg(S_{G \backslash N_G[v]}/I_{G \backslash N_G[v]})$ and both
	$G \backslash v$ and $G \backslash N_G[v]$ are induced subgraphs of $G$.

	Let $V=X \sqcup Y$ with  $X=\{x_1,x_2,\dots,x_n\}$, $Y=\{y_1,y_2,\dots,y_n\}$  and  $u=x_{\ell}$.  Suppose $N_G(u)=\{v\}$ and  $N_G(v)=\{x_{i_1},\dots,x_{i_t}\}$ with  $1\leq i_1<i_2<\dots<i_t=\ell$.  Two cases are discussed below:
	
(i) If $N_G(v)= X$, then  $G \backslash N_G[v]=Y \backslash \{v\}$  consists of  isolated points. Hence, $I_{G \backslash N_G[v]}=0$, which  implies that $\reg(S_G/J)=\reg(S_{G \backslash N_G[v]}/I_{G \backslash N_G[v]})=0$.

 (ii) If $N_G(v)\subsetneq X$, then  $G \backslash N_G[v]$ is the disjoint union of a graph $H$ and the isolated set $\{y_{i_1},y_{i_2},\dots,y_{i_{t-1}}\}$, where  $H=G\backslash\{x_{i_1},\dots,x_{i_t},y_{i_1},\dots,y_{i_t}\}$. So by  Lemma \ref{depth-reg}(\ref{depth-reg-2})  we have
 \[
 \reg(S_G/J)=\reg(S_{G \backslash N_G[v]}/I_{G \backslash N_G[v]})=\reg(S_H/I_H)=\vartheta(H),
 \]
since  $H$ is a C-M bipartite graph by Lemma \ref{cm subgraph}(\ref{cm subgraph-2}). Let $M=\{e_1,e_2,\dots,e_{\vartheta(H)}\}$ be an induced matching of $H$ and $e=\{u,v\}$. Since $V(H)\cap N_G[v]=\emptyset$ and  $u$ is  a leaf with $N_G(u)=\{v\}$,
we have  $M\cap e=\emptyset$, which  implies  $M\sqcup \{e\}$ is  an induced matching of $G$.  Hence, in this case, $\reg(S_G/I_{G})\ge \vartheta(H)+1$,  establishing
the claim.
\end{proof}

\begin{Lemma}
\label{reg of S/J}
 Let $G=(V,E)$ be a  C-M bipartite graph  with  a leaf $u$. Let $N_G(u)=\{v\}$ with  $\deg_G(v) \geq 2$. Suppose  $V=X \sqcup Y$  where $X=\{x_1,x_2,\dots,x_n\}$ and $Y=\{y_1,y_2,\dots,y_n\}$.
  Let $u=y_{i_1}$,  $N_G(v)=\{y_{i_1},y_{i_2},\dots,y_{i_t}\}$ with $1\leq i_1 <i_2<\dots<i_t \leq n$ and $w=y_{i_t}$. Suppose also that $J=(N_G(w))+I_{G \backslash N_G[w]}$,  $K=(w)+I_{G \backslash w}$.  If $\vartheta(G \backslash v) = \vartheta(G)-1$, then
\begin{enumerate}
 \item[(1)]  $\reg(S_G/J) \leq \reg(S_G/I_{G})-2$;
  \item[(2)] $\reg(S_G/K) = \reg(S_G/I_{G})$.
  \end{enumerate}
\end{Lemma}

\begin{proof}Since $N_G(v)=\{y_{i_1},y_{i_2},\dots,y_{i_t}\}$ with $1\leq i_1 <i_2<\dots<i_t \leq n$, we get that $x_{i_t}$ is a leaf and $N_G(x_{i_t})=\{w\}$. 	Let $N_G(w)=\{x_{j_1},x_{j_2},\dots,x_{j_m}\}$ with $1\le j_1<j_2<\dots<j_m=i_t$.
	It follows that $\reg(S_G/K)=\reg(S_{G \backslash w}/I_{G \backslash w})$ and $\reg(S_G/J) \leq\reg(S_G/I_{G})-1$ from Lemma \ref{reg of S/J-1}.
	
(1) Conversely, if $\reg(S_G/J)=\reg(S_G/I_{G})-1$, then by Lemma \ref{depth-reg}, we have
\[
\reg(S_G/J)=\vartheta(G)-1. \hspace{3.5cm} (\dag)
\]
We distinguish between the following two cases:

(i) If $N_G(w)= X$, then  $G \backslash N_G[w]=Y \backslash \{w\}$ is consists of  isolated points.
It follows that $\reg(S_G/J)=\reg(S_{G \backslash N_G[w]}/I_{G \backslash N_G[w]})=0$,
which implies $\vartheta(G)=1$ by formula (\dag). Note that $\deg_G(v) \geq 2$, we have $\reg(S_{G \backslash \{u,v\}}/I_{G \backslash \{u,v\}}) \geq 1$. On the other hand, since
 $G \backslash \{u,v\}$ is a C-M bipartite graph by  Lemma \ref{cm subgraph}(\ref{cm subgraph-1}), it follows from  Lemma \ref{depth-reg}(\ref{depth-reg-2})  that  $\reg(S_{G \backslash \{u,v\}}/I_{G \backslash \{u,v\}})=\vartheta(G \backslash \{u,v\})=\vartheta(G \backslash v)=\vartheta(G)-1$.  Thus $\vartheta(G)\ge 2$, a contradiction to $\vartheta(G)=1$.

(ii) If $N_G(w) \subsetneq X$.  Then $G \backslash N_G[w]$ is the disjoint union of $H$ and the  isolated set $\{y_{j_1},y_{j_2},\dots,y_{j_{m-1}}\}$, where $H=G\backslash \{x_{j_1},\dots,x_{j_m},y_{j_1},\dots,y_{j_m}\}$ is C-M bipartite graph  by  Lemma \ref{cm subgraph}(\ref{cm subgraph-2}).  So by  Lemma \ref{depth-reg}(\ref{depth-reg-2})  we have
\[
\reg(S_G/J)=\reg(S_{G \backslash N_G[u]}/I_{G \backslash N_G[u]})=\reg(S_H/I_H)=\vartheta(H).
\]
It follows that  $\vartheta(H)=\vartheta(G)-1$  by formula (\dag). Let $M=\{e_1,e_2,\dots,e_{\vartheta(G)-1}\}$ be an induced matching of $H$ and $e=\{x_{i_t},w\}$.
Since $V(H)\cap N_G[w]=\emptyset$ and  $x_{i_t}$ is  a leaf with  $N_G(x_{i_t})=\{w\}$,
we have  $M\cap e=\emptyset$, which implies  that $M\sqcup \{e\}$ is an induced matching of $G\backslash v$. Note that the size of $M\sqcup \{e\}$ is $\vartheta(G)$,  which contradicts $\vartheta(G \backslash v) = \vartheta(G)-1$.

(2) Let $M=\{e_1,e_2,\dots,e_{\vartheta(G)}\}$ be any  induced matching of $G$.
Claim:  $\{u,v\}\in M$.

Indeed, if $\{u,v\}\notin M$, then $v$ is a vertex of $e_i$ for some $i\in [\vartheta(G)]$, since $\vartheta(G \backslash v)=\vartheta(G)-1$. Let $e_i=\{v,y_{i_\ell}\}$,  then $x_{i_\ell}\cap e_j=\emptyset$ for any $e_j\in M$ with $j\ne i$. Conversely,  if $x_{i_\ell}\cap e_j\neq \emptyset$  for some $e_j\in M$ with $j\ne i$, then we choose edge $e=\{x_{i_\ell},y_{i_\ell}\}$, thus $e\cap e_i=\{y_{i_\ell}\}$
and $e\cap e_j=\{x_{i_\ell}\}$, which contradicts with $M$ being an  induced matching  of $G$. This implies that
$x_{i_\ell}$  cannot be a vertex of any edge in the set $M\backslash \{e_i\}$.  Substituting edge $\{x_{i_\ell},y_{i_\ell}\}$ for $e_i$ yields an induced match of $G \backslash v$.
 Consequently, $\vartheta(G \backslash v)=\vartheta(G)$,  which contradicts with $\vartheta(G \backslash v) = \vartheta(G)-1$.

 Note that  $w=y_{i_t}$ and $N_G(v)=\{y_{i_1},y_{i_2},\dots,y_{i_t}\}$, thus  $w\in N_G(v)$. Since  $\{u,v\}\in M$, $w$ cannot be a vertex of any edge in $M$ by the definition of induced matching, which implies that $\vartheta(G \backslash w)=\vartheta(G)$. It follows that
$\reg(S_G/K)=\reg(S_{G \backslash w}/I_{G \backslash w})=\reg(S_{G \backslash \{w,x_{i_t}\}}/I_{G \backslash \{w,x_{i_t}\}})=\vartheta(G \backslash \{w,x_{i_t}\})=\vartheta(G \backslash w)=\vartheta(G)= \reg(S_G/I_{G})$ by Lemmas
 \ref{cm subgraph}(\ref{cm subgraph-1}) and  \ref{depth-reg}(\ref{depth-reg-2}).
\end{proof}

\begin{Remark}
\label{aaa}
Let $G=(V,E)$ be a  C-M bipartite graph   with  a leaf $u$. Suppose $N_G(u)=\{v\}$ with  $\deg_G(v) \geq 2$.
If
$\vartheta(G \backslash v)<\vartheta(G)$, then we can obtain  $\reg(S_G/I_G) \geq 2$ by similar arguments as  the subcase
    $(ii)$ in the proof of Lemma \ref{reg of S/J-1}.
\end{Remark}

\begin{Theorem}
\label{V(G_2)=1}
Let $G=(V,E)$ be a  C-M bipartite graph  with  a leaf $u$. Let $N_G(u)=\{v\}$, then
\[
\reg(S_{G\backslash u}/I_{G\backslash u})=\reg(S_{G}/I_{G})-s,
\]
where $s= \begin{cases}
	0, & \text{if $\vartheta(G \backslash v) =\vartheta(G)$},\\
	1, & \text{otherwise.}
\end{cases}$
\end{Theorem}

\begin{proof}
Let $V=X \sqcup Y$ with $X=\{x_1,\dots,x_n\}$, $Y=\{y_1,\dots,y_n\}$ and $u=y_{\ell}$. Assume that $N_G(u)=\{v\}$ and $N_G(v)=\{y_{i_1},y_{i_2},\dots,y_{i_t}\}$ with  $\ell= i_1 <i_2 < \dots <i_t \leq n$.  Let $H=G\backslash u$.

 (1) If $\vartheta(G \backslash v) = \vartheta(G)$. In this case, we suppose
 $J=(N_H(v))+I_{H \backslash N_H[v]}$,  $K=(v)+I_{H \backslash v}$, then $I_{H}=J\cap K$ and  $H\backslash \{x_{i_1},x_{i_2},y_{i_2},\dots,x_{i_t},y_{i_t}\}=G\backslash \{x_{i_1},y_{i_1},
\dots,x_{i_t},y_{i_t}\}$ is a C-M bipartite graph by  Lemma \ref{cm subgraph}(\ref{cm subgraph-2}). Thus by Lemma \ref{reg of S/J-1}(1), we have $\reg(S_H/J)=\reg(S_{H\backslash N_H[v]}/I_{H\backslash N_H[v]})=\reg(S_{G\backslash N_G[v]}/I_{G\backslash N_G[v]}) \leq \reg(S_G/I_G)-1$.
  Meanwhile, $H \backslash v=G\backslash \{u,v\}$ is a C-M bipartite graph  by Lemma \ref{cm subgraph} (\ref{depth-reg-2}). Thus $\reg(S_H/K)=\reg(S_{H\backslash v}/I_{H\backslash v})=\reg(S_{G\backslash \{u,v\}}/I_{G\backslash \{u,v\}})=\vartheta(G \backslash \{u,v\})=\vartheta(G \backslash v)=\reg(S_G/I_G)$.
 By Lemmas \ref{lem:direct_sum}(2),  \ref{exact}(\ref{exact-4}),
  \ref{decomposition}(\ref{decomposition-4}) and the  exact sequence (\ref{eqn:SES-1}),
 we obtain $\reg(S_{G\backslash u}/I_{G\backslash u})=\reg(S_{G}/I_{G})$.

 $(2)$ If $\vartheta(G \backslash v) \neq \vartheta(G)$. We prove the statement by induction on $\deg_G(v)$.

 (i) If $\deg_G(v)=1$, then $H$ is the disjoint union of  $G\backslash \{u,v\}$ and isolated vertice $v$, and  $G\backslash \{u,v\}$ is a C-M bipartite graph by  Lemma
 \ref{cm subgraph}(\ref{cm subgraph-1}). Thus,
  $\reg(S_H/I_{H})=\reg(S_{G\backslash \{u,v\}}/I_{G\backslash \{u,v\}})=\vartheta(G \backslash \{u,v\})=\vartheta(G \backslash v)=\vartheta(G)-1= \reg(S_{G}/I_{G})-1$.

(ii) If $\deg_G(v) \geq 2$. Let $w=y_{i_t}$ and $N_G(w)=\{x_{j_1},x_{j_2},\dots,x_{j_m}\}$, where $1 \leq j_1 <j_2<\dots<j_m=i_t $. In this case, let $J=(N_H(w))+I_{H \backslash N_H[w]}$,  $K=(w)+I_{H \backslash w}$, then $I_{H}=J\cap K$. We divide into the following two cases for $H \backslash N_H[w]$:

(a) If $N_H[w]=X$, then $H \backslash N_H[w] =Y\backslash \{u,w\}$ consists of isolated points. Hence, $I_{H \backslash N_H[w]}=0$, which implies  $\reg(S/J)=\reg(S_{H \backslash N_H[u]}/I_{H \backslash N_H[u]})=0\le \reg(S_{G}/I_{G})-2$ by Remark \ref{aaa}.

(b) If $N_H[w]\subsetneq X$, then $H \backslash N_H[w]$ is the disjoint union of $H'$ and isolated set $\{y_{j_1},y_{j_2},\dots, y_{j_{t-1}}\} \backslash u$,
 where $H'= H\backslash \{x_{j_1},y_{j_1},\dots,x_{j_m},y_{j_m}\}$. By Lemma \ref{cm subgraph}(\ref{cm subgraph-2}), we have
 $H\backslash \{x_{j_1},y_{j_1},\dots,x_{j_m},y_{j_m}\}=G\backslash \{x_{j_1},y_{j_1},\dots,x_{j_m},y_{j_m}\}$ is a C-M bipartite graph. Thus $\reg(S_H/J)=\reg(S_{H\backslash N_H[w]}/I_{H\backslash N_H[w]})
 =\reg(S_{G\backslash N_G[w]}/I_{G\backslash N_G[w]})\\ \leq \reg(S_{G}/I_{G})-2$ by Lemma \ref{reg of S/J}(1).

In order to  compute  $\reg(S_H/K)$, we apply induction on $\deg_G(v)$.

If $\deg_G(v)= 2$, $H \backslash w$ is the disjoint union of $G\backslash \{x_{i_1},y_{i_1},x_{i_t},y_{i_t}\}$ and isolated set $\{x_{i_1},x_{i_t}\}$, then $\reg(S_H/K) =\reg(S_{H\backslash w}/I_{H\backslash w})=\reg(S_{G\backslash \{x_{i_1},y_{i_1},x_{i_t},y_{i_t}\}}/I_{G\backslash \{x_{i_1},y_{i_1},x_{i_t},y_{i_t}\}})$ $=\vartheta(G \backslash \{x_{i_1},y_{i_1},x_{i_t},y_{i_t}\})=\vartheta(G)-1=\reg(S_G/I_G)-1$.

Now assume that $\deg_G(v)\geq 3$. Let $G'=G\backslash \{x_{i_t},y_{i_t}\}$, then $G'$ is a  C-M bipartite graph with a leaf $x_{i_{t-1}}$ and  $\deg_{G'}(v)=\deg_G(v)-1$. Meanwhile, $H \backslash w$ is the disjoint union of $G'\backslash u$ and isolated point $x_{i_t}$.  Thus
 \begin{align*}
\reg(S_H/K)&=\reg(S_{H\backslash w}/I_{H\backslash w})=\reg(S_{G'\backslash u}/I_{G'\backslash u})=\reg(S_{G'}/I_{G'})-1\\
 &=\vartheta(G \backslash \{x_{i_t},y_{i_t}\})-1=\reg(S_{G}/I_{G})-1.
\end{align*}
 Note that $I_{H}=J\cap K$, we obtain $\reg(S_{G\backslash u}/I_{G\backslash u})=\reg(S_{G}/I_{G})-s$ by applying Lemmas \ref{lem:direct_sum} and  to  \ref{exact}(\ref{exact-4}) the  exact sequence (\ref{eqn:SES-1}).
\end{proof}

\medskip
\section{the $\circ$ operation and the  $*$ operation}
\label{sec:operation}
In this section,  we will study some   graphs  obtained from  Cohen-Macaulay bipartite graphs by the $\circ$ operation or the  $*$ operation. The main task of this section is to give some exact formulas for the depth and regularity  of the  edge ideals of such graphs.  We start by
recalling from \cite{BMS,SZ} the two aforementioned special gluing operations.

\begin{Definition}
    \label{circ_*_operations}
    For $i= 1,2$, let $G_i$ be a graph with
    a leaf $u_i$. Furthermore, let  $N_G(u_i)=\{v_i\}$  with $\deg_{G_i}(v_i)\ge 2$.
    \begin{enumerate}
   	 \item[(1)]  Let $G$ be a graph obtained from $G_1$ and $G_2$ by  first removing the leaves $u_1, u_2$, and then identifying the vertices $v_1$ and $v_2$.  In this case, we say that $G$ is obtained from $G_1$ and $G_2$ by the \emph{$\circ$ operation} and write $G=(G_1,u_1) \circ (G_2,u_2)$ or simply $G=G_1 \circ G_2$. If $v_1$ and $v_2$ are identified as the vertex $v$ in $G$, then we also write $G= G_1\circ_v G_2$. Unless otherwise specified, when we perform the $\circ$ operation in this way, we always implicitly assume that neither $G_1$ nor $G_2$ is the path graph $P_2$ of two vertices.

      \item[(2)]  Let $H$ be the graph obtained from $G_1$ and $G_2$ by identifying the vertices $u_1$ and $u_2$. In this case, we say that $H$ is obtained from $G_1$ and $G_2$ by the \emph{$*$ operation} and write $H=(G_1,u_1)*(G_2,u_2)$ or simply $H= G_1 * G_2$. If we denote the identified vertex in $H$ by $u$, then we also write $H= G_1 *_u G_2$.
    \end{enumerate}
\end{Definition}

\begin{Theorem}
 \label{thm:depth_CM_CM_circ}
Let $G=(G_1,u_1) \circ (G_2,u_2)$, where each $G_i$ is a  C-M bipartite graph  with  a leaf $u_i$. Let $N_G(u_i)=\{v_i\}$, then
\[
\depth(S_G/I_{G})=\depth(S_{G_1}/I_{G_1})+\depth(S_{G_2}/I_{G_2})-s,
\]
where $s= \begin{cases}
	1, & \text{if $\deg_{G_i}(v_i)=1$ for all $i\in [2]$},\\
	2, & \text{otherwise.}
\end{cases}$

\end{Theorem}
\begin{proof}Let  $V(G_i)=\{x_{i,1},x_{i,2},\dots,x_{i,n_i}\} \sqcup \{y_{i,1},y_{i,2},\dots,y_{i,n_i}\}$ for $i\in [2]$. By symmetry,   we can assume that every $u_i=x_{i,j_i}$ for some $j_i\in [n_i]$, and  $v_i=y_{i,j_i}$ is the only neighbor point of $x_{i,j_i}$ in $G_i$. Suppose  $y_{1,j_1}$ and $y_{2,j_2}$ are identified as $v$ in $G$ by the $\circ$ operation. We distinguish into three cases:

(I) If $\deg_{G_i}(v_i)=1$ for all $i\in [2]$, then $G$   is the disjoint union of $G_1\backslash \{u_1,v_1\}$, $G_2\backslash  \{u_2,v_2\}$ and an isolated vertex $v$. Thus, by Lemmas  \ref{cm subgraph}(\ref{cm subgraph-1}) and \ref{depth-reg}(\ref{depth-reg-1}), we have
\begin{align*}
\depth(S_G/I_{G})&=1+[\depth(S_{G_1}/I_{G_1})-1]+[\depth(S_{G_2}/I_{G_2})-1]\\
&=\depth(S_{G_1}/I_{G_1})+\depth(S_{G_2}/I_{G_2})-1.
\end{align*}

(II) If $\deg_{G_1}(v_1)\ge 2$ and  $\deg_{G_2}(v_2)=1$, then $G=(G_1\backslash u_1)\sqcup (G_2\backslash  \{u_2,v_2\})$. Thus, by Lemmas  \ref{cm subgraph}(\ref{cm subgraph-1}), \ref{depth-reg}(\ref{depth-reg-1}) and
\ref{subgraph-depth}, we have
\begin{align*}
\depth(S_G/I_{G})&=\depth(S_{G_1\backslash u_1}/I_{G_1\backslash u_1})+\depth(S_{G_2\backslash  \{u_2,v_2\}}/I_{G_2\backslash  \{u_2,v_2\}})\\
&=\depth(S_{G_1}/I_{G_1})+\depth(S_{G_2}/I_{G_2})-2.
\end{align*}

(III) If $\deg_{G_i}(v_i)\ge 2$ for all $i\in [2]$. We will prove the statement by induction on $n_2$.
 When $n_2=2$. Let  $N_{G_2}(v_2)=\{x_{2,1},x_{2,2}\}$, where $x_{2,2}=u_2$.
In this case, $G=G_1 \cup _{x_{2,1}} P_2$  is the  clique sum of $G_1$ and a path $P_2$ with vertex set $\{x_{2,1},y_{2,1}\}$. Set $J=(N_{G}(y_{2,1}))+I_{G\backslash N_G[y_{2,1}]}$,  $K=(y_{2,1})+I_{G \backslash y_{2,1}}$. Note that $G\backslash N_G[y_{2,1}]=G_1 \backslash u_1$ and $G \backslash y_{2,1} =G_1$, we obtain  $\depth(S_G/J)=1+\depth(S_{G\backslash N_G[y_{2,1}]}/I_{G\backslash N_G[y_{2,1}]})=1+\depth(S_{G_1 \backslash u_1}/I_{G_1 \backslash u_1})=\depth(S_{G_1}/I_{G_1})$ by Lemma \ref{subgraph-depth},  and $\depth(S_G/K)=\depth(S_{G \backslash y_{2,1}}/I_{G \backslash y_{2,1}})=\depth(S_{G_1}/I_{G_1})$.   Using Lemmas \ref{lem:direct_sum}(1), \ref{exact}(\ref{exact-2}) and   \ref{decomposition}(\ref{decomposition-3}) to the following exact sequence
\begin{equation}
	0\longrightarrow \frac{S_G}{J \cap K}\longrightarrow \frac{S_G}{J} \oplus \frac{S_G}{K}\longrightarrow \frac{S_G}{J+K} \longrightarrow 0
 \label{eqn:SES-2},
\end{equation}	
we get the desired  regularity results.

In the following,  we assume that $n_2\ge 3$.
Let  $N_{G_2}(v)=\{x_{2,\ell_1},x_{2,\ell_2},\dots,x_{2,\ell_s}\}$ with $1\leq \ell_1<\ell_2<\dots<\ell_s=j_2$, then  $\deg_{G_2}(y_{2,\ell_1})=1$ by Theorem  \ref{cm bipart}(\ref{cm bipart-2}).
In this case, let $N_{G}(x_{2,\ell_1})=\{y_{2,k_1},y_{2,k_2},\dots,y_{2,k_t}\}$  with $\ell_1=k_1<k_2<\dots<k_t\leq n_2$ and $k_s=j_2$ for some $s\in[t]$.
Set $J=(N_{G}(x_{2,\ell_1}))+I_{G\backslash N_G[x_{2,\ell_1}]}$,  $K=(x_{2,\ell_1})+I_{G \backslash x_{2,\ell_1}}$.
 We  distinguish between  the following two cases:

(A1)   If $|N_{G_2}(v)|=2$, then $s=2$ and   $G\backslash x_{2,\ell_1}$   is the disjoint union of $G_1\backslash u_1$, $G_2\backslash  \{x_{2,\ell_1},y_{2,\ell_1},x_{2,j_2},y_{2,j_2}\}$ and an isolated vertex $y_{2,\ell_1}$.
So by Lemmas \ref{depth-reg}(\ref{depth-reg-1}) and \ref{subgraph-depth},
 we have
\begin{align*}
\depth(S_G/K)&=1+\depth(S_{G_1\backslash u_1}/I_{G_1\backslash u_1})+\depth(S_{H_1}/I_{H_1})\\
&=1+[\depth(S_{G_1}/I_{G_1})-1]+[\depth(S_{G_2}/I_{G_2})-2]\\
&=\depth(S_{G_1}/I_{G_1})+\depth(S_{G_2}/I_{G_2})-2
\end{align*}
where $H_1=G_2\backslash  \{x_{2,\ell_1},y_{2,\ell_1},x_{2,j_2},y_{2,j_2}\}$.

Meanwhile,   $G\backslash N_G[x_{2,\ell_1}]$  takes one of the following forms:
\begin{enumerate}
 \item[(a)]  $G\backslash N_G[x_{2,\ell_1}]=H\sqcup H_1$, where $H=G_1\backslash \{u_1,v_1\}$;
  \item[(b)] $G\backslash N_G[x_{2,\ell_1}]$  is the disjoint union of $H$, $H_2$ and the isolated set  $\{x_{2,k_2},\dots,x_{2,k_t}\}\backslash \{u_2\}$,  where $H_2=G_2\backslash  \{x_{2,k_1},\dots ,x_{2,k_t},y_{2,k_1}, \dots,y_{2,k_t}\}$.
  \end{enumerate}
Note that $H$, $H_1$ and  $H_2$ are  C-M bipartite graphs. We consider two subcases:

 (i) If $G\backslash N_G[x_{2,\ell_1}]$  is of form (a), then we get by Lemma \ref{depth-reg}(\ref{depth-reg-1}) that
\begin{align*}
 \depth(S_G/J)&=1+\depth(S_{H}/I_{H})+\depth(S_{H_1}/I_{H_1})\\
&=1+[\depth(S_{G_1}/I_{G_1})-1]+[\depth(S_{G_2}/I_{G_2})-2]\\
&=\depth(S_{G_1}/I_{G_1})+\depth(S_{G_2}/I_{G_2})-2.
\end{align*}
  (ii) If  $G\backslash N_G[x_{2,\ell_1}]$  is of form (b), then we also have
  \begin{align*}
	\depth(S_G/J)&=1+\depth(S_{H}/I_{H})+\depth(S_{H_2}/I_{H_2})+(t-2)\\
&=1+[\depth(S_{G_1}/I_{G_1})-1]+(\depth(S_{G_2}/I_{G_2})-t)+(t-2)\\
&=\depth(S_{G_1}/I_{G_1})+\depth(S_{G_2}/I_{G_2})-2.
\end{align*}

(A2)  If $|N_{G_2}(v)|\ge 3$, then  $G\backslash x_{2,\ell_1}$  is the disjoint union of $G'$ and an isolated vertex $y_{2,\ell_1}$, where  $G'=(G_1,u_1)\circ (G'_2,u_2)$ and $G'_2=G_2\backslash  \{x_{2,\ell_1},y_{2,\ell_1}\}$.
In this case, we have $|N_{G'_2}(v)|=|N_{G_2}(v)|-1$. Thus, by induction and
 Lemma \ref{depth-reg}(\ref{depth-reg-1}), we have
\begin{align*}
 \depth(S_H/K)&=1+\depth(S_{G'}/I_{G'})\\
 &=1+[\depth(S_{G_1}/I_{G_1})+\depth(S_{G'_2}/I_{G'_2})-2]\\
 &=\depth(S_{G_1}/I_{G_1})+[\depth(S_{G_2}/I_{G_2})-1]-1\\
 &=\depth(S_{G_1}/I_{G_1})+\depth(S_{G_2}/I_{G_2})-2.
\end{align*}
At the same time,   $G\backslash N_G[x_{2,\ell_1}]$  has one of the following forms:
\begin{enumerate}
 \item[(c)]  $G\backslash N_G[x_{2,\ell_1}]=H\sqcup H_1$, where $H=G_1\backslash \{u_1,v_1\}$;
  \item[(d)] $G\backslash N_G[x_{2,\ell_1}]$  is the disjoint union of $H$, $H_2$ and  isolated set  $\{x_{2,k_2},\dots,x_{2,k_t}\}\backslash \{u_2\}$,  where $H_2=G_2\backslash  \{x_{2,k_1},y_{2,k_1},x_{2,k_2},y_{2,k_2},\dots,x_{2,k_t},y_{2,k_t}\}$.
  \end{enumerate}
Applying the analysis as   $(a)$ and $(b)$ of in  case (1) above, we obtain
\[
 \depth(S_G/J)=\depth(S_{G_1}/I_{G_1})+\depth(S_{G_2}/I_{G_2})-2.
 \]
In summary,  the desired result follows from  Lemmas 1.2(1),   \ref{exact}(\ref{exact-2}), \ref{decomposition}  and  exact sequence (2).
\end{proof}

	\begin{Example}
	The following are two examples that satisfy the two conditions in Theorem \ref{thm:depth_CM_CM_circ}, respectively.
	
  \begin{figure}[!htb]
                \begin{minipage}{.4\linewidth}
                    \centering
                    \begin{tikzpicture}[thick, scale=1.0, every node/.style={scale=0.98}]]
                       \draw[solid](0,2)--(0,0);
                    \draw[solid](0,2)--(2,0);

                     \draw[solid](1,2)--(1,0);

                       \draw[solid](2,2)--(2,0);

                      \shade [shading=ball, ball color=black] (0,2) circle (.07) node [above] {\scriptsize$x_{1,1}$};
                      \shade [shading=ball, ball color=black] (1,2) circle (.07) node [above] {\scriptsize$u_1$};
                      \shade [shading=ball, ball color=black] (2,2) circle (.07) node [above] {\scriptsize$x_{1,3}$};

                       \shade [shading=ball, ball color=black] (0,0) circle (.07) node [below] {\scriptsize$y_{1,1}$};
                       \shade [shading=ball, ball color=black] (1,0) circle (.07) node [below] {\scriptsize$v_1$};
                       \shade [shading=ball, ball color=black] (2,0) circle (.07) node [below] {\scriptsize$y_{1,3}$};

                    \end{tikzpicture}
                    \centerline{$G_1$}
                \end{minipage}\hfill
         \begin{minipage}{.1\linewidth}
  \centering
             \begin{tikzpicture}[thick, scale=1.0, every node/.style={scale=0.98}]]
               \draw[solid](0,2)--(0,0);

                \shade [shading=ball, ball color=black] (0,2) circle (.07) node [above] {\scriptsize$u_2$};

                \shade [shading=ball, ball color=black] (0,0) circle (.07) node [below] {\scriptsize$v_2$};

             \end{tikzpicture}
                    \centerline{$G_2$}
        \end{minipage}\hfill
      \begin{minipage}{.1\linewidth}
      \centering
                    \begin{tikzpicture}[thick, scale=1.0, every node/.style={scale=0.98}]]
                   \draw[->](0,1)--(1.5,1);
                    \end{tikzpicture}
        \end{minipage}\hfill
         \begin{minipage}{.4\linewidth}
                    \centering
                    \begin{tikzpicture}[thick, scale=1.0, every node/.style={scale=0.98}]]
                        \draw[solid](0,2)--(0,0);
                        \draw[solid](0,2)--(2,0);

                        \draw[solid](2,2)--(2,0);

                         \shade [shading=ball, ball color=black] (0,2) circle (.07) node [above] {\scriptsize$x_{1,1}$};
                        \shade [shading=ball, ball color=black] (2,2) circle (.07) node [above] {\scriptsize$x_{1,3}$};

                        \shade [shading=ball, ball color=black] (0,0) circle (.07) node [below] {\scriptsize$y_{1,1}$};
                       \shade [shading=ball, ball color=black] (1,0) circle (.07) node [below] {\scriptsize$v=v_1$};
                       \shade [shading=ball, ball color=black] (2,0) circle (.07) node [below] {\scriptsize$y_{1,3}$};

                    \end{tikzpicture}
                    \centerline{$G_1 \circ_v G_2$}
        \end{minipage}\hfill
                \caption{}
            \end{figure}
            \begin{figure}[!htb]
                \begin{minipage}{0.3\linewidth}
                     \centering
                    \begin{tikzpicture}[thick, scale=1.0, every node/.style={scale=0.98}]]
                        \draw[solid](0,2)--(0,0);
                      \draw[solid](0,2)--(1,0);

                        \draw[solid](1,2)--(1,0);

                         \shade [shading=ball, ball color=black] (0,2) circle (.07) node [above] {\scriptsize$x_{1,1}$};
                         \shade [shading=ball, ball color=black] (1,2) circle (.07) node [above] {\scriptsize$u_1$};

                      \shade [shading=ball, ball color=black] (0,0) circle (.07) node [below] {\scriptsize$y_{1,1}$};
                      \shade [shading=ball, ball color=black] (1,0) circle (.07) node [below] {\scriptsize$v_1$};

                    \end{tikzpicture}
                    \subcaption*{$G_1$}
                \end{minipage}\hfill
              \begin{minipage}{0.2\linewidth}
                    \centering
                    \begin{tikzpicture}[thick, scale=1.0, every node/.style={scale=0.98}]]
                        \draw[solid](0,2)--(0,0);
                       \draw[solid](0,2)--(2,0);

                         \draw[solid](1,2)--(1,0);
                       \draw[solid](1,2)--(2,0);

                       \draw[solid](2,2)--(2,0);

                       \shade [shading=ball, ball color=black] (0,2) circle (.07) node [above] {\scriptsize$x_{2,1}$};
                         \shade [shading=ball, ball color=black] (1,2) circle (.07) node [above] {\scriptsize$x_{2,2}$};
                         \shade [shading=ball, ball color=black] (2,2) circle (.07) node [above] {\scriptsize$u_2$};

                         \shade [shading=ball, ball color=black] (0,0) circle (.07) node [below] {\scriptsize$y_{2,1}$};
                      \shade [shading=ball, ball color=black] (1,0) circle (.07) node [below] {\scriptsize$y_{2,2}$};
                      \shade [shading=ball, ball color=black] (2,0) circle (.07) node [below] {\scriptsize$v_2$};
                    \end{tikzpicture}
                    \centerline{$G_2$}
              \end{minipage}\hfill
              \begin{minipage}{.1\linewidth}
                    \centering
                    \begin{tikzpicture}[thick, scale=1.0, every node/.style={scale=0.98}]]
                   \draw[->](0,1)--(1.5,1);
                    \end{tikzpicture}

             \end{minipage}\hfill
               \begin{minipage}{.4\linewidth}
                     \centering
                    \begin{tikzpicture}[thick, scale=1.0, every node/.style={scale=0.98}]]
                        \draw[solid](0,2)--(0,0);
                      \draw[solid](0,2)--(1.4,0);

                      \draw[solid](0.7,2)--(0.7,0);
                      \draw[solid](0.7,2)--(1.4,0);

                      \draw[solid](2.1,2)--(1.4,0);

                      \draw[solid](2.1,2)--(2.1,0);

                     \shade [shading=ball, ball color=black] (0,2) circle (.07) node [above] {\scriptsize$x_{2,1}$};
                        \shade [shading=ball, ball color=black] (0.7,2) circle (.07) node [above] {\scriptsize$x_{2,2}$};
                        \shade [shading=ball, ball color=black] (2.1,2) circle (.07) node [above] {\scriptsize$x_{1,1}$};

                       \shade [shading=ball, ball color=black] (0,0) circle (.07) node [below] {\scriptsize$y_{2,1}$};
                       \shade [shading=ball, ball color=black] (0.7,0) circle (.07) node [below] {\scriptsize$y_{2,2}$};
                       \shade [shading=ball, ball color=black] (1.4,0) circle (.07) node [below] {\scriptsize$v$};
                       \shade [shading=ball, ball color=black] (2.1,0) circle (.07) node [below] {\scriptsize$y_{1,1}$};

                    \end{tikzpicture}
                     \centerline{$G_1 \circ_v G_2$}
                \end{minipage}\hfill
                \caption{}
            \end{figure}

	Let $G=G_1 \circ G_2$. In Figure $3$, $G$ is the disjoint union of an isolated vertice $v$ and a path of length   $3$ with vertice set $\{x_{1,1},y_{1,1},y_{1,3},x_{1,3}\}$, thus
$\depth(S_{G}/I_{G})=3$ by Lemma \ref{path}.
In Figure $4$, let  $J=(x_{N_G(x_{2,1})})+I_{G \backslash N_G[x_{2,1}]}$, $K=(x_{2,1})+I_{G \backslash x_{2,1}}$. In this case, $\depth(S_G/J)=1+\depth(S_{G\backslash N_G[x_{2,1}]}/I_{G\backslash N_G[x_{2,1}]})=3$ and  $\depth(S_G/K)=3$ by Lemma \ref{path}. It follows from Lemmas $\ref{lem:direct_sum}(1)$, $\ref{exact}$, $\ref{decomposition}$ and the exact sequence $(\ref{eqn:SES-2})$ that $\depth(S_G/I_{G})=3$.	
\end{Example}

\begin{Theorem}\label{yang}
Let $G=(G_1,u_1) \circ (G_2,u_2)$, where each $G_i$ is a C-M bipartite graph with a leaf $u_i$.  Let $N_{G_i}(u_i)=\{v_i\}$.
If  $t=|\{i:\vartheta(G_i \backslash v_i) \neq \vartheta(G_i)\}|$, then
 \[
 \reg(S_G/I_G)=\reg(S_{G_1}/I_{G_1})+\reg(S_{G_2}/I_{G_2})-t.
 \]
\end{Theorem}
\begin{proof}
First,  $t \in \{0,1,2\}$ by the definition of $t$. Let $V_i$ be the vertex set of $G_i$ and  $V_i=X_i \sqcup Y_i$ be a bipartition of $V_i$ with $X_i=\{x_{i,1},x_{i,2},\dots,x_{i,n_i}\}$, $Y_i=\{y_{i,1},y_{i,2},\dots,y_{i,n_i}\}$ for $i\in[2]$.
By symmetry, let $N_{G_1}(v_1)=\{x_{1,i_1},x_{1,i_2},\dots,x_{1,i_m}\}$ and $N_{G_2}(v_2)=\{y_{2,j_1},y_{2,j_2},\dots,y_{2,j_t}\}$, where $u_1=x_{1,i_m}$, $u_2=y_{2,j_1}$,
 $1 \leq i_1 <i_2< \dots <i_m \leq n_1$ and $1 \leq j_1 <j_2< \dots <j_t \leq n_2$.  Suppose that $v_1$ and $v_2$ are identified as $v$ in $G$ by the $\circ$ operation and
$ N_{G_2}(y_{2,j_t})=\{x_{2,h_1},x_{2,h_2},\dots,x_{2,h_s}\}$ with $1 \leq h_1 <h_2< \dots <h_s = j_t$.
We divide into the following two cases:

(I) If $t=2$, then $\vartheta(G_i \backslash v_i) \neq \vartheta(G_i)$ for all $i \in [2]$. Now we prove  the formulas for the  regularity  of $S_G/I_G$ by induction on $\deg_{G_2}(v_2)$.
If $\deg_{G_2}(v_2)=1$, then $G=(G_1,u_1) \circ (G_2,u_2)$ is the disjoint union of $G_1 \backslash u_1$ and $G_2 \backslash \{u_2, v_2\}$.
If $G_2 \backslash \{u_2, v_2\}=\emptyset$, then the desired result follows from Theorem \ref{V(G_2)=1}. Now, we assume that $G_2 \backslash \{u_2, v_2\}\ne \emptyset$. In this case,
 $G_2 \backslash \{u_2, v_2\}$ is a C-M bipartite graph  by Lemma  \ref{cm subgraph}(\ref{cm subgraph-1}). It follows from Lemmas \ref{sum}(\ref{sum-1}), \ref{depth-reg}(\ref{depth-reg-2}) and \ref{V(G_2)=1} that
\begin{align*}
\reg(S_G/I_G)&=\reg(S_{G_1\backslash u_1}/I_{G_1\backslash u_1})+\reg(S_{G_2\backslash \{u_2,v_2\}}/I_{G_2\backslash \{u_2,v_2\}})\\
 &=(\reg(S_{G_1}/I_{G_1})-1)+\vartheta(G_2 \backslash \{u_2,v_2\})\\
 &=(\reg(S_{G_1}/I_{G_1})-1)+\vartheta(G_2 \backslash v_2)\\
 &=(\reg(S_{G_1}/I_{G_1})-1)+(\vartheta(G_2)-1)\\
 &=\reg(S_{G_1}/I_{G_1})+\reg(S_{G_2}/I_{G_2})-2.
\end{align*}
Assume that  $\deg_{G_2}(v_2)\geq 2$ and  the regularity statement holds for $\deg_{G_2}(v_2)-1$. In this case, let $w=y_{2,j_t}$, then  $x_{2,j_t}$ is a leaf of $G_2$. Choose $J=(N_G(w))+I_{G \backslash N_G[w]}$,  $K=(w)+I_{G \backslash w}$. thus $G\backslash w$ is the disjoint union of $G_1 \circ (G_2\backslash \{x_{2,j_t},w\})$ and isolated vertice $x_{2,j_t}$. Let $G'_2=G_2\backslash \{x_{2,j_t},w\}$, then $\deg_{G'_2}(v_2)= \deg_{G_2}(v_2)-1$. By the induction hypothesis, we have
\begin{align*}
\reg(S_G/K)&=\reg(S_{G\backslash w}/I_{G\backslash w})=\reg(S_{G_1 \circ G'_2}/I_{G_1 \circ G'_2})\\
 &=\reg(S_{G_1}/I_{G_1})+\reg(S_{G'_2}/I_{G'_2})-2\\
 &=\reg(S_{G_1}/I_{G_1})+\reg(S_{G_2}/I_{G_2})-2.
\end{align*}
where the last equality holds because of $t=2$ and  Lemma \ref{reg of S/J}(2).

In order to  compute  $\reg(S_G/J)$, we consider the induced subgraph $G \backslash N_G[w]$  of $G$.
We distinguish into the following two cases:

(a) If $N_{G_2}(w) =X_2$, then  $G \backslash N_G[w]$ is the disjoint union of $G_1 \backslash \{u_1,v_1\}$ and isolated set $Y_2\backslash \{u_2,w\}$, and  $G_1 \backslash \{u_1,v_1\}$ is a C-M bipartite graph  by Lemma  \ref{cm subgraph}(\ref{cm subgraph-1}).   It follows that
\begin{align*}
\reg(S_G/J)&=\reg(S_{G \backslash N_G[w]}/I_{G \backslash N_G[w]})=\reg(S_{G_1 \backslash \{u_1,v_1\}}/I_{G_1 \backslash \{u_1,v_1\}})\\
 &=\vartheta(G_1\backslash \{u_1,v_1\})=\vartheta(G_1\backslash v_1)=\reg(S_{G_1}/I_{G_1})-1.
\end{align*}

(b) If $N_{G_2}(w)  \subsetneq X_2$, then  $G \backslash N_G[w]$ is the disjoint union of $G_1 \backslash \{u_1,v_1\}$,
$H$ and the isolated set $\{y_{2,h_1},y_{2,h_2},\dots,y_{2,h_{s-1}}\} \backslash \{u_2\}$, where $H=G_2 \backslash \{x_{2,h_1},y_{2,h_1},\dots, x_{2,h_s},y_{2,h_s}\}$.
By Lemmas \ref{depth-reg}(\ref{depth-reg-2}) and \ref{reg of S/J}(1), we obtain
\begin{align*}
\reg(S_G/J)&=\reg(S_{G \backslash N_G[w]}/I_{G \backslash N_G[w]})=\reg(S_{G_1 \backslash \{u_1,v_1\}}/I_{G_1 \backslash \{u_1,v_1\}})+\reg(S_H/I_H)\\
  &=\vartheta(G_1 \backslash v_1)+\reg(S_H/I_H)\\
 &=(\reg(S_{G_1}/I_{G_1})-1)+\reg(S_H/I_H)\\
 &\leq (\reg(S_{G_1}/I_{G_1})-1)+(\reg(S_{G_2}/I_{G_2})-2)\\
  &=\reg(S_{G_1}/I_{G_1})+\reg(S_{G_2}/I_{G_2})-3.
\end{align*}
where the penultimate inequality holds by Lemma \ref{reg of S/J}(1).
Using Lemmas \ref{lem:direct_sum}(2), \ref{exact}(\ref{exact-4}) and   \ref{decomposition} to the  exact sequence (\ref{eqn:SES-2}),
we get the desired  regularity results.

(II) If $t \leq 1$. In this case, we choose  $J=(N_G(v))+I_{G \backslash N_G[v]}$,  $K=(v)+I_{G \backslash v}$. Thus $G \backslash N_G[v]$ is the disjoint union of $G'_1, G'_2$
and the  isolated set $\{y_{1,i_1},\dots,y_{1,i_{m-1}},x_{2,j_2},\\
\dots,x_{2,j_t}\}$, where $G'_1=G_1\backslash \{x_{1,i_1},y_{1,i_1},\dots,x_{1,i_m},y_{1,i_m}\}$ $G'_2=G_2\backslash \{x_{2,j_1},y_{2,j_1},\dots,x_{2,j_t},\\
y_{2,j_t}\}$. By Lemmas  \ref{cm subgraph}(\ref{cm subgraph-2}) and \ref{depth-reg}(\ref{depth-reg-2}), we get
\begin{align*}
\reg(S_G/J)&=\reg(S_{G \backslash N_G[v]}/I_{G \backslash N_G[v]})=\reg(S_{G'_1}/I_{G'_1})+\reg(S_{G'_2}/I_{G'_2})\\
 &\leq (\reg(S_{G_1}/I_{G_1})-1)+(\reg(S_{G_2}/I_{G_2})-1)\\
  &=\reg(S_{G_1}/I_{G_1})+\reg(S_{G_2}/I_{G_2})-2.
\end{align*}
where the penultimate inequality holds by Lemma \ref{reg of S/J-1}.
On the other hand, $G\backslash v$ is the disjoint union  of $G_1\backslash \{u_1,v_1\}$ and $G_2\backslash \{u_2,v_2\}$. By Lemmas  \ref{cm subgraph}(\ref{cm subgraph-1}) and \ref{depth-reg}(\ref{depth-reg-2}), we get
 \begin{align*}
\reg(S_G/K)&=\reg(S_{G \backslash v}/I_{G \backslash v})\\
 &=\reg(S_{G_1 \backslash \{u_1,v_1\}}/I_{G_1 \backslash \{u_1,v_1\}})+\reg(S_{G_2 \backslash \{u_2,v_2\}}/I_{G_2 \backslash \{u_2,v_2\}})\\
  &=\vartheta(G_1 \backslash v_1)+\vartheta(G_2 \backslash v_2).
\end{align*}
Thus, if $t=0$, then $\vartheta(G_i \backslash v_i)=\vartheta(G_i)$ for all $i\in [2]$, Thus
$\reg(S/K)=\reg(S_{G_1}/I_{G_1})+\reg(S_{G_2}/I_{G_2})$. If $t=1$, then $\vartheta(G_1 \backslash v_1)=\vartheta(G_1)$ and $\vartheta(G_1 \backslash v_1)\ne\vartheta(G_1)$,
or vice versa. Thus $\reg(S/K)=\reg(S_{G_1}/I_{G_1})+\reg(S_{G_2}/I_{G_2})-1$.

 Applying Lemma \ref{lem:direct_sum}, \ref{exact}(\ref{exact-4}) and  \ref{decomposition} to the  exact sequence (\ref{eqn:SES-2}), we obtain that $\reg(S_G/I_G)=\reg(S_{G_1}/I_{G_1})+\reg(S_{G_2}/I_{G_2})-t$.
\end{proof}

\begin{Theorem}
Let $G=(G_1,u_1) * (G_2,u_2)$, where each $G_i$ is a C-M bipartite graph with a leaf $u_i$. Let $N_{G_i}(u_i)=\{v_i\}$.   Let $t=|\{i:\vartheta(G_i \backslash v_i) \neq \vartheta(G_i)\}|$. Then
\[\reg(S_G/I_G)=\reg(S_{G_1}/I_{G_1})+\reg(S_{G_2}/I_{G_2})-s\]
where $s= \begin{cases}
	0, & \text{if $t \leq 1$},\\
	1, & \text{if $t=2$.}
\end{cases}$

\end{Theorem}
\begin{proof} First,  $t \in \{0,1,2\}$ by the definition of $t$.
Suppose  $u_1$ and $u_2$ are identified as $u$ in $G$ by the $*$ operation. Let $N_{G_2}(v_2)=\{y_{j_1},y_{j_2},\dots,y_{j_m}\}$ with $u_2=y_{j_1}$, where $1 \leq j_1<\dots<j_m\leq n_2$.  Then $G \backslash N_G[v_2]$ is the disjoint union of $G_1 \backslash u_1$, $G_2\backslash \{x_{j_1},y_{j_1},\dots,x_{j_m},y_{j_m}\}$ and isolated set $\{x_{j_2},\dots,x_{j_m}\}$, and $G \backslash v_2$ is the disjoint union of $G_1$ and $G_2 \backslash \{u_2,v_2\}$. We divide into the following two cases:

(1) If $t \leq 1$, then $\vartheta(G_i \backslash v_i) = \vartheta(G_i)$ for some $i\in[2]$. By symmetry, we  assume $\vartheta(G_2 \backslash v_2) = \vartheta(G_2)$. In this case, we choose $J=(N_G(v_2))+I_{G \backslash N_G[v_2]}$,  $K=(v_2)+I_{G \backslash v_2}$. Let $H=G_2\backslash \{x_{j_1},y_{j_1},\dots,x_{j_m},y_{j_m}\}$, then  by Lemmas \ref{lem:induced-reg},  \ref{reg of S/J-1}(1) and  \ref{depth-reg}(\ref{depth-reg-2}), we have
\begin{align*}
\reg(S_G/J)&=\reg(S_{G \backslash N_G[v_2]}/I_{G \backslash N_G[v_2]})=\reg(S_{G_1\backslash u_1}/I_{G_1\backslash u_1})+\reg(S_{H}/I_{H})\\
 &\leq \reg(S_{G_1}/I_{G_1})+(\reg(S_{G_2}/I_{G_2})-1)\\
  &=\reg(S_{G_1}/I_{G_1})+\reg(S_{G_2}/I_{G_2})-1,\\
\reg(S_G/K)&=\reg(S_{G \backslash v_2}/I_{G \backslash v_2})=\reg(S_{G_1}/I_{G_1})+\reg(S_{G_2 \backslash \{u_2,v_2\}}/I_{G_2 \backslash \{u_2,v_2\}})\\
 &=\reg(S_{G_1}/I_{G_1})+\vartheta(G_2 \backslash \{v_2\})\\
  &=\reg(S_{G_1}/I_{G_1})+\reg(S_{G_2}/I_{G_2}).
\end{align*}
 Applying Lemmas \ref{lem:direct_sum}, \ref{exact}(\ref{exact-4}) and  \ref{decomposition}(\ref{decomposition-2}) to the  exact sequence (\ref{eqn:SES-2}), we obtain the expected results.

 (2) If $t=2$, then $\vartheta(G_i \backslash v_i) \neq \vartheta(G_i)$ for all $i
 \in [2]$.  In this case, we choose  $J=(N_G(v_2))+I_{G \backslash N_G[v_2]}$,  $K=(v_2)+I_{G \backslash v_2}$. Thus, by Theorem \ref{V(G_2)=1}, we have
 \begin{align*}
\reg(S_G/J)&=\reg(S_{G \backslash N_G[v_2]}/I_{G \backslash N_G[v_2]})=\reg(S_{G_1\backslash u_1}/I_{G_1\backslash u_1})+\reg(S_{H}/I_{H})\\
 &= (\reg(S_{G_1}/I_{G_1})-1)+(\reg(S_{G_2}/I_{G_2})-1)\\
  &=\reg(S_{G_1}/I_{G_1})+\reg(S_{G_2}/I_{G_2})-2.
\end{align*}
where the penultimate equality holds  by the proof of Lemma \ref{reg of S/J}(2).
 Meanwhile, by  Lemmas  \ref{cm subgraph}(\ref{cm subgraph-1}) and  \ref{depth-reg}(\ref{depth-reg-2}), we have
 \begin{align*}
\reg(S_G/K)&=\reg(S_{G \backslash v_2}/I_{G \backslash v_2})=\reg(S_{G_1}/I_{G_1})+\reg(S_{G_2 \backslash \{u_2,v_2\}}/I_{G_2 \backslash \{u_2,v_2\}})\\
   &=\reg(S_{G_1}/I_{G_1})+\vartheta(G_2 \backslash v_2)=\reg(S_{G_1}/I_{G_1})+\vartheta(G_2)-1\\
  &=\reg(S_{G_1}/I_{G_1})+\reg(S_{G_2}/I_{G_2})-1.
\end{align*}
Again applying Lemmas \ref{lem:direct_sum}, \ref{exact}(\ref{exact-4}) and  \ref{decomposition} to the  exact sequence (\ref{eqn:SES-2}), we obtain the  desired result.
\end{proof}

For two graph $G_1$ and $G_2$, let  their clique sum $G_1\cup_v  G_2$ be a union of graphs $G_1$ and $G_2$ such that $V(G_1)\cap V(G_2)=\{v\}$.
\begin{Lemma}
\label{Lemma c-m and P *}
Let $G=G_1 \cup _{u} P_2$  be the  clique sum of a C-M bipartite graph $G_1$ and a path $P_2$ with vertex set $\{u,v_2\}$, where  $u$ is a leaf of $G_1$ and $N_{G_1}(u)=\{v_1\}$.
 Then
\[\depth(S_G/I_G)=\depth(S_{G_1}/I_{G_1})\]
\end{Lemma}
\begin{proof}
If $\deg_{G_1}(v_1)=1$, then $G$ is the disjoint union of $G_1\backslash \{u_1,v_1\}$ and a path of length $3$ with vertice set $\{v_1,v_2,u\}$.  Thus, by  Lemma \ref{depth-reg}(\ref{depth-reg-1}) and Lemma \ref{path}, we get
$\depth(S_G/I_G)=1+\frac{|V(G_1)|-2}{2}=\depth(S_{G_1}/I_{G_1})$, since $G_1\backslash \{u,v_1\}$ is a C-M bipartite graph  by Lemma  \ref{cm subgraph}(\ref{cm subgraph-1}).

If $\deg_{G_1}(v_1) \ge 2$, then we choose  $J=(N_G(v_2))+I_{G \backslash N_G[v_2]}$,  $K=(v_2)+I_{G \backslash v_2}$. In this case, $G \backslash v_2 =G_1$ and  $G \backslash N_G[v_2]=G_1 \backslash u$. Then by Lemma \ref{subgraph-depth}, we obtain that
\begin{align*}
\depth(S_G/J)&=1+\depth(S_{G \backslash N_G[v_2]}/I_{G \backslash N_G[v_2]})\\
&=1+\depth(S_{G_1\backslash u}/I_{G_1\backslash u})\\
&=1+(\depth(S_{G_1}/I_{G_1})-1)\\
&=\depth(S_{G_1}/I_{G_1}),\\
\depth(S_G/K)&=\depth(S_{G\backslash v_2}/I_{G\backslash v_2})=\depth(S_{G_1}/I_{G_1}).
\end{align*}
Using Lemmas Lemmas \ref{lem:direct_sum}, \ref{exact}(\ref{exact-2}) and \ref{decomposition}  to the  exact sequence (\ref{eqn:SES-2}), we get $\depth(S_G/I_G)=\depth(S_{G_1}/I_{G_1})$.
\end{proof}

\begin{Theorem}
 \label{thm:depth_CM_CM_*}
Let $G=(G_1,u_1) * (G_2,u_2)$, where each $G_i$ is a  C-M bipartite graph  with  a leaf $u_i$. Let $N_{G_i}(u_i)=\{v_i\}$. Then
\[
\depth(S_G/I_{G})=\depth(S_{G_1}/I_{G_1})+\depth(S_{G_2}/I_{G_2})-1,
\]
\end{Theorem}
\begin{proof}
For any $i\in[2]$, let $V(G_i)=X_{i} \sqcup Y_{i}$ with $X_{i}=\{x_{i,1},\dots,x_{i,n_i}\}$, $Y_{i}=\{y_{i,1},\dots,y_{i,n_i}\}$ and $u_i=x_{i,j_i}$ for some $j_i \in [n_i]$. Then
$N_{G_i}(u_i)=\{y_{i,j_i}\}$. Let $v_i=y_{i,j_i}$ and  $N_{G_1}(v_1)=\{x_{1,k_1},\dots,x_{1,k_t}\}$ with $1\leq k_1<\dots <k_t=j_1$. We divide into the following two cases:

(1) If $\deg_{G_i}(v_i)=1$ for some $i \in [2]$, then we  assume   $\deg_{G_2}(v_2) = 1$ by symmetry. Then $G_2=P_2\sqcup(G_2\backslash  \{u_2,v_2\})$, which implies $G=(G_1 \cup _{u_1} P_2)\sqcup (G_2\backslash  \{u_2,v_2\})$. By Lemmas \ref{sum}(\ref{sum-2}), \ref{depth-reg}(\ref{depth-reg-1}), and \ref{Lemma c-m and P *}, we have
\begin{align*}
\depth(S_G/I_{G})&=\depth(S_{G_1 \cup _{u_1} P_2}/I_{G_1 \cup _{u_1} P_2})+\depth(S_{G_2\backslash  \{u_2,v_2\}}/I_{G_2\backslash  \{u_2,v_2\}})\\
&=\depth(S_{G_1}/I_{G_1})+\depth(S_{G_2}/I_{G_2})-1.
\end{align*}

(2) If $\deg_{G_i}(v_i) \geq 2$ for all  $i \in [2]$, then we choose $J=(N_G(v_1))+I_{G \backslash N_G[v_1]}$, $K=(v_1)+I_{G \backslash v_1}$. In this case,  $G\backslash v_1$ is the disjoint union of $G_2$ and $G_1\backslash \{u_1,v_1\}$. Thus  by Lemma \ref{depth-reg}(\ref{depth-reg-1}), we get
\begin{align*}
\depth(S_G/K)&=\depth(S_{G_2}/I_{G_2})+\depth(S_{G_1 \backslash \{u_1,v_1\}}/I_{G_1 \backslash \{u_1,v_1\}})\\
&=\depth(S_{G_2}/I_{G_2})+\depth(S_{G_1}/I_{G_1})-1.
\end{align*}
In order to compute the depth of $S_G/J$, we distinguish into the following two cases:

(i) If $N_G(v_1)=X_1$, then $G \backslash N_G[v_1]$ is the disjoint union of  $G_2 \backslash u_2$ and isolated set $Y_1 \backslash v_1$. By Lemmas  \ref{sum} (\ref{sum-2}), \ref{depth-reg}(\ref{depth-reg-1}) and \ref{subgraph-depth},  we have
\begin{align*}
\depth(S_G/J)&=1+\depth(S_{G \backslash N_G[v_1]}/I_{G \backslash N_G[v_1]})\\
&=1+\depth(S_{G_2\backslash u_2}/I_{G_2\backslash u_2})+(n_1-1)\\
&=\depth(S_{G_1}/I_{G_1})+\depth(S_{G_2}/I_{G_2})-1.
\end{align*}

(ii) If $N_G(v_1) \subsetneq X_1$, then $G \backslash N_G[v_1]$ is the disjoint union of $G_1 \backslash \{x_{1,k_1},y_{1,k_1},\dots, \\
x_{1,k_t}, y_{1,k_t}\}$, $G_2\backslash u_2$, and isolated set $\{y_{1,k_1},y_{1,k_2},\dots,y_{1,k_{t-1}}\}$. Note that $H=G_1 \backslash \{x_{1,k_1},\\
y_{1,k_1},\dots, x_{1,k_t}, y_{1,k_t}\}$  is a C-M bipartite graph by Lemma  \ref{cm subgraph}(\ref{cm subgraph-2}).
 Thus, by Lemmas  \ref{depth-reg}(\ref{depth-reg-1}) and
\ref{subgraph-depth}, we have
\begin{align*}
\depth(S_G/J)&=1+\depth(S_{G \backslash N_G[v_1]}/I_{G \backslash N_G[v_1]})\\
&=1+\depth(S_{H}/I_{H})+\depth(S_{G_2\backslash u_2}/I_{G_2\backslash u_2})+(t-1)\\
&= 1+\frac{|V(G_1)|-2t}{2}+[\depth(S_{G_2}/I_{G_2})-1]+(t-1)\\
&= \frac{|V(G_1)|}{2}+\depth(S_{G_2}/I_{G_2})-1\\
&=\depth(S_{G_1}/I_{G_1})+\depth(S_{G_2}/I_{G_2})-1.
\end{align*}
 Applying  Lemmas \ref{lem:direct_sum}, \ref{exact}(\ref{exact-2}) and  \ref{decomposition} to the  exact sequence (\ref{eqn:SES-2}), we obtain that
$\depth(S_G/I_{G})=\depth(S_{G_1}/I_{G_1})+\depth(S_{G_2}/I_{G_2})-1$. We finish the proof.
\end{proof}

\medskip
\hspace{-6mm} {\bf Acknowledgments}

 \vspace{3mm}
\hspace{-6mm}  This research is supported by the Natural Science Foundation of Jiangsu Province (No. BK20221353) and the
    foundation of the Priority Academic Program Development of Jiangsu Higher
    Education Institutions. The authors are grateful to the computer algebra systems CoCoA \cite{Co} for providing us with a large number of examples.

\medskip
\hspace{-6mm} {\bf Data availability statement}

\vspace{3mm}
\hspace{-6mm}  The data used to support the findings of this study are included within the article.

\medskip

\end{document}